\documentclass[a4wide,11pt]{article}
\usepackage{bbm}
\usepackage{geometry}
\usepackage{graphicx}
\usepackage{authblk}
\usepackage{amssymb}
\usepackage{epstopdf}
\usepackage{rotating}
\usepackage{epigraph} 
\usepackage[sans]{dsfont}
\usepackage[T1]{fontenc}
\usepackage[english]{babel}
\usepackage{latexsym}
\usepackage{mathrsfs}
\usepackage{hyperref}
\usepackage{amscd}
\usepackage{sidecap}
\usepackage{caption}
\usepackage{subcaption}

\usepackage{graphicx}
\usepackage{color}
\usepackage{algorithm}
\usepackage{algorithmic} 
\usepackage{tikz}
\usepackage{amsmath}
\usepackage{amsfonts}
\numberwithin{equation}{section}
\usepackage{enumerate}
\usepackage{amsthm}
\usepackage[numbers,sort]{natbib}
\usepackage{array}

\usepackage[pagewise]{lineno}

\newcolumntype{P}[1]{>{\centering\arraybackslash}p{#1}}
\newcolumntype{M}[1]{>{\centering\arraybackslash}m{#1}}
\graphicspath{{pics/}}

\newcommand{\R}{\mathbb R}
\newcommand{\Hinf}{\mathcal{H}_\infty}

\newcommand{\ddt}{\frac{d}{dt}}

\newcommand{\Id}{\textrm{Id}}

\newcommand{\ko}{k_o}
\newcommand{\kd}{k_d}

\def\be#1\ee{\begin{equation}#1\end{equation}}

\newtheorem{thm}{Theorem}[section]

\newtheorem{lem}[thm]{Lemma} 
\newtheorem{prop}[thm]{Proposition}

\newtheorem{remark}{Remark}[section] 
\newcommand{\bq}{\begin{equation}}
	\newcommand{\eq}{\end{equation}}

\title{Robust feedback stabilization of interacting multi-agent systems under uncertainty}
\date{}

\author{Giacomo Albi\thanks{Department of Computer Science, University of Verona, Str. Le Grazie 15, Verona, I-37134, Italy }, \ Michael Herty\thanks{IGPM, RWTH Aachen University, Templergraben, 55, D-52062 Aachen, Germany}, \ and Chiara Segala\thanks{IGPM, RWTH Aachen University, Templergraben, 55, D-52062 Aachen, Germany
		\\
		MH and CS thank the Deutsche Forschungsgemeinschaft (DFG, German Research Foundation) for the financial support through 320021702/GRK2326, 333849990/IRTG-2379, HE5386/18-1,19-2,22-1,23-1 and under Germany's Excellence Strategy EXC-2023 Internet of Production 390621612.  GA thanks the Italian Ministry of Instruction, University and Research (MIUR) to support this research with funds coming from PRIN Project 2017 (No. 2017KKJP4X entitled “Innovative numerical methods for evolutionary partial differential equations and applications”).
	}
}

\begin{document}
	\maketitle
	\begin{abstract}
		We consider control strategies for large-scale interacting agent systems under uncertainty. The particular focus is on the design of robust controls that allow to bound the variance of the controlled system over time. To this end we consider $\mathcal{H}_\infty$ control strategies on the agent and mean field description of the system. We show  a bound on the  $\mathcal{H}_\infty$  norm for a stabilizing controller independent on the number of agents.  Furthermore, we compare the new control with existing approaches to treat uncertainty by generalized polynomial chaos expansion. Numerical results are presented for one-dimensional and two-dimensional agent systems.
	\end{abstract}
	
	\par
	
	\par

	\section{Introduction}
	We consider the mathematical modelling and control of phenomena of collective dynamics under uncertainties. These phenomena have been studied in several fields such as socio-economy, biology, and robotics where systems of interacting particles are given by self-propelled particles, such as animals and robots, see e.g. \cite{bellomo20review, MR2974186, MR2165531, MR2580958,Giselle}. Those particles interact according to a possibly nonlinear model,  encoding various social rules as attraction, repulsion, and alignment. A particular feature of such models is their rich dynamical structure, which includes different types of emerging patterns, including consensus, flocking, and milling \cite{MR2887663, MR2247927,cucker2007emergent,d2006self,motsch2014heterophilious}. Understanding the impact of control inputs in such complex systems is of great relevance for applications. Results in this direction allow to design optimized actions such as collision-avoidance protocols for swarm robotics \cite{CKPP19,KPAsurvey15,MR3157726,Meurer}, pedestrian evacuation in crowd dynamics \cite{MR3308728,dyer2009leadership}, supply chain policies \cite{MR2844776,degond2007network}, the quantification of interventions in traffic management \cite{MR3948232,han2017resolving,stern2018dissipation} or in opinion dynamics \cite{Garnier,goddard2022noisy}.
	Further, the introduction of uncertainty in the mathematical modelling of real-world phenomena seems to be unavoidable for applications, since often at most statistical information of the modelling parameters is available. The latter has typically been estimated from experiments or derived from heuristic observations \cite{ballerini2008empirical,bongini2017inferring,katz2011inferring}. To produce effective predictions and to describe and understand physical phenomena, we may incorporate parameters reflecting the uncertainty in the interaction rules, and/or external disturbances \cite{CiCP-25-508}.
	
	Here, we are concerned with the robustness of controls influencing the evolution of a collective motion of an interacting agent system. The controls we are considering are aimed to stabilize the system's dynamic under external uncertainty. From a mathematical point of view, a description of self-organized models is provided by complex system theory, where the overall dynamics are depicted by a large-scale system of ordinary differential equations (ODEs).
	\par 
	More precisely,  we  consider the control of high-dimensional dynamics accounting  $N$ agents with state $v_i(t,\theta) \in \mathbb{R}^d,\, i=1,\ldots,N$, evolving according to 
	\begin{equation}\label{lin_dynamics_intro}
		\ddt{v}_i(t,\theta) = \sum_{j=1}^N a_{ij}(v_j(t,\theta)-v_i(t,\theta))+ u_i(t,\theta) +\sum_{k=1}^Z \theta_k\,, \qquad v_i(0)=v_{i}^0,
	\end{equation}
	where $A=[a_{ij}]\in\R^{N\times N}$ defines the nature of pairwise interaction among agents, and  $\theta=(\theta_1,\ldots,\theta_Z)^\top\in\R^{Z\times d}$ is a random input vector with a given probability density distribution on $Z$ as  $\rho\equiv\rho_1\otimes\ldots\otimes\rho_Z$.
	The control signal $u_i(t,\theta)\in \R^d$ is designed to stabilize the state toward a target state  $\bar v\in\R^{N\times d}$, and its action is influenced by the random parameter $\theta$. This is also due to the fact, that later we will be interested in closed--loop or feedback controls on the state $(v_1.\dots,v_N)$ that in turn dependent on the unknown parameter $\theta.$ 
	\par 
	Of particular interest will be controls designed via minimization of linear quadratic (parametric) regulator functional such as
	\begin{equation}\label{eq:LQR_intro}
		\min_{u(\cdot,\theta)} {J}(u;v^0) := \int_0^{+\infty} \exp( - r \tau )  \left[v^\top Q v +\nu  u^\top R u  \right]\,d\tau,
	\end{equation}
	with $Q$ positive semi-definite matrix of order $N$,  $R$ positive definite matrix of order $N$ and $r$ is a discount factor. In this case, the linear quadratic dynamics allow for an optimal control $u^*$ stabilising the desired state $v_d=0$,  expressed in feedback form, and obtained by solving the associated matrix Riccati -equations. Those aspects will be also addressed in more detail below.
	\par 
	In order to assess the performances of controls, and quantify their robustness we propose estimates using the concept of  $\Hinf$ control. In this setting different approaches have  been studied in the context of $\Hinf$ control and  applied to first-order and higher-order multiagent systems, see e.g. \cite{luo2021event,lin2010robust,mo2013finite,liu2019robust,liujia2011robust}, in particular for an interpretation of  $\Hinf$ as dynamic games we refer to  \cite{bacsar2008h}. 
	Here we will study an approach based on the derivation of sufficient conditions in terms of linear matrix inequalities (LMIs) for the $\Hinf$ control problem. In this way, consensus robustness will be ensured for a general feedback formulation of the control action.  Additionally, we consider the large--agent limit and show that the robustness is guaranteed independently of the number of agents. 
	
	%
	Furthermore, we will discuss the numerical realization of system \eqref{lin_dynamics_intro} employing uncertainty quantification techniques. In general, at the numerical level, techniques for uncertainty quantification can be classified into non-intrusive and intrusive methods. In a non-intrusive approach, the underlying model is solved for fixed samples with deterministic schemes, and statistics of interest are determined by numerical quadrature, typical examples are Monte-Carlo and stochastic collocation methods \cite{dimarco2017uncertainty,xiu2010numerical}.
	While in the intrusive case, the dependency of the solution on the stochastic input is described as a truncated series expansion in terms of orthogonal functions. Then, a new system is deduced that describes the unknown coefficients in the expansion. One of the most popular techniques of this type is based on stochastic Galerkin (SG) methods. In particular, generalized polynomial chaos (gPC) gained increasing popularity in uncertainty quantification (UQ), for which spectral convergence on the random field is observed under suitable regularity assumptions
	\cite{dimarco2017uncertainty,hu2017uncertainty,hu2015stochastic,xiu2010numerical}.
	The methods, here developed, make use of the stochastic Galerkin (SG) for the microscopic dynamics while in the mean-field case we combine SG in the random space with a Monte Carlo method in the physical variables.
	
	The manuscript is organized as follows, in Section \ref{sec:control}  we introduce the problem setting and propose different feedback control laws;   in Section \ref{sec:Hinf} we reformulate the problem in the setting of $\Hinf$ control and provide conditions for the robustness of the controls in the microscopic and mean-field case. Section \ref{sec:numerics_noise} is devoted to the description of numerical strategies for the simulation of the agent systems, and to different numerical experiments, which assess the performances and compare different methods.

	\section{Control of  interacting agent system with uncertainties}\label{sec:control}
	
	The following notation is introduced with the control of high-dimensional systems of interacting agents with random inputs. We consider the evolution of $N$ agents with state $v(t,\theta)\in\mathbb{R}^{ N\times d}$ as follows
	\begin{equation}\label{noisyModel}
		\ddt{v}_i(t,\theta) = \frac{1}{N}\sum_{j=1}^N \bar{p}(v_j(t,\theta)-v_i(t,\theta)) + u_i(t,\theta) + \sum_{k = 1}^Z \theta_k\,
	\end{equation}
	with deterministic initial data $v_i(0)=v_{i}^0$ for  $i=1,\ldots,N$,  and where $\theta_k  \in \Omega_k\subseteq \mathbb{R}^{d}$ for $k=1,\ldots,Z$ are random inputs, distributed according to a compactly supported probability density $\rho\equiv \rho_1 \otimes \dots \otimes \rho_Z$, i.e.,  $\rho_k(\theta)\geq 0$ a.e.,  $\textrm{supp}(\rho_k)\subseteq \Omega_k $  and  $\int_{\Omega_k} \rho_k(\theta)\, d\theta=1$. For simplicity, we also assume that the random inputs have zero average  $\mathbb{E}[\theta_k] =0$.
	The  control signal $u(t,\theta)\in\R^{N\times d}$ is designed minimizing the (parameterized) objective
	\begin{equation}\label{eq:u_star}
		u^*(\cdot,\theta) = \arg\min_{u(\cdot,\theta)} {J}(u;v^0) := \int_0^{+\infty} \exp(-r \tau) \left( \frac{1}{N}\sum_{j=1}^N ( \vert v_j(\tau,\theta) - \bar{v} \vert^2 + \nu \vert u_j(\tau,\theta)\vert^2)  \right) \,d\tau,
	\end{equation}
	with $\nu>0$ being a penalization parameter for the control energy, the norm $\vert \cdot \vert$ being the usual Euclidean norm in $\mathbb{R}^d$. The discount factor $\exp(-r \tau)$ is introduced to have a well-posed integral.
	\par 
	We assume that $\bar{v}$ is a prescribed consensus point, namely, in the context of this work we are interested in reaching a consensus velocity $\bar{v}\in\R^d$ such that $v_{1}=\ldots=v_{N}=\bar{v}$, and w.l.o.g. we can assume $\bar{v} = 0$. Note that $\bar{v}=0$ is also the steady state of the dynamics in absence of disturbances. Hence, we may view $u(\cdot,\theta)$ as a stabilizing control of the zero steady state of the system. Furthermore, we will be interested in feedback controls $u.$ 
	
	Recall that the (deterministic) linear model \eqref{noisyModel}, without uncertainties, allows a feedback stabilization  by solving the resulting optimal control problem through a Riccati equations \cite{HePaSt15,albi2021momentdriven,albi2021gradient}.
	The functional $J$ in \eqref{eq:u_star}, in absence of disturbances, reads as follows
	\begin{align*}
		J(u;v^0) = \int_0^{+\infty} \exp(-r \tau) \left( v^\top Q v+ \nu u^\top R u \right)\, dt
	\end{align*}
	where $Q \equiv R = \frac{1}{N}\textrm{Id}_N$. 
	In this case the controlled dynamics \eqref{noisyModel}  is reformulated in a matrix-vector notation
	\begin{equation}\label{eq:linctrl}
		\ddt v(t)= A v(t) + B u(t),  \qquad  u(t) = - \frac{N}{\nu} K v(t),
	\end{equation}
	with $B=\textrm{Id}_N$ the identity matrix of order $N$,  and
	\begin{align}
		(A)_{ij} =  
		\begin{cases} 
			&a_d=\frac{\bar{p}(1-N)}{N},\qquad i=j,\\
			&a_o=\frac{\bar{p}}{N},\qquad\qquad i\neq j.\\
		\end{cases}
	\end{align}
	The matrix $K$ associated to feedback form of the optimal control has to fulfilll the Riccati matrix-equation  of the following form 
	\begin{equation}\label{eq:Riccati_inf}
		0= - r K +KA+A^\top K-\frac{N}{\nu} KK + Q.
	\end{equation}
	For a general linear system, we need to solve the $N\times N$ equations to find $K$, which can be costly for large-scale agent-based dynamics. However, we can use the same argument of \cite{albi2021momentdriven}
	and exploit the symmetric structure of the Laplacian matrix $A$ to reduce the algebraic Riccati equation.
	Unlike in \cite{albi2021momentdriven}, where they investigate the case with finite terminal time, here we state the following proposition for the infinite horizon case with discount factor $r$
	\begin{prop}[Properties of the Algebraic Riccati Equation (ARE)] 
		For the linear dynamics \eqref{eq:linctrl}, the solution of the Riccati equation \eqref{eq:Riccati_inf} reduces to the solution of
		\begin{equation}\label{kd_ko0}
			\begin{split}
				0 &= -r k_d -\frac{2\bar{p}(N-1)}{N}({k}_d -k_o) - \frac{N}{\nu}\left({k}_d^2+(N-1){k}_o^2\right) + \frac{1}{N}, \,
				\\
				0&=  -r k_o + \frac{ 2\bar{p}}{N}(k_d-k_o) - \frac{N}{\nu}\left(2{k}_d {k}_o+(N-2){k}_o^2\right). 
			\end{split}
		\end{equation}
		The entries $(i,j)$ of the matrix $K$ of the algebraic Riccati equation \eqref{eq:Riccati_inf} is  given by 
		$$(K)_{ij}=\delta_{ij}k_d+(1-\delta_{ij})k_o.$$
	\end{prop}
	
	In order to allow the limit of infinitely many agents $N\to\infty$, we introduce the following scalings
	\begin{equation*}
		{k}_d \leftarrow N k_d, \quad {k}_o \leftarrow N^2 k_o,\quad \alpha(N) = \frac{N-1}{N},
	\end{equation*}
	and keeping the same notation also for the scaled variables $\kd,\ko$, the system \eqref{kd_ko0} reads 
	\begin{equation}\label{kd_ko}
		\begin{split}
			0&= -r k_d -2\bar{p}\alpha(N)\left(\kd - \frac{\ko}{N}\right) - \frac{1}{\nu}\left(\kd^2+\frac{\alpha(N)}{N}\ko^2\right) + 1, 
			\\
			0&=   -r k_o + 2\bar{p}\left(\kd - \frac{\ko}{N}\right) - \frac{1}{\nu}\left(2\kd \ko+\alpha(N)\ko^2-\frac{1}{N}\ko^2\right).
		\end{split}
	\end{equation}	
	\par 
	The previous considerations motivate to extend formula \eqref{eq:linctrl} to the parametric case \eqref{noisyModel}.  Hence, in the presence of parametric uncertainty  the feedback control is written explicitly as follows
	\begin{equation} \label{control}
		u_i(t,\theta) = - \frac{1}{\nu} \left( K v(t,\theta) \right)_i = - \frac{1}{\nu}\left(\left(k_d-\frac{k_o}{N}\right) v_i(t,\theta) + \frac{k_o}{N} \sum_{j=1}^N v_j(t,\theta)\right).
	\end{equation}
	The question arises if the given feedback is robust with respect to the uncertainties $\theta$. In the following,  we will provide a measure for the robustness of control \eqref{control} in the framework of $\Hinf$ control. Some additional remarks allow generalizing formula \eqref{control}. 
	
	\begin{remark}[Non-zero average]
		In presence of general uncertainties with known expectations, we  modify the control \eqref{control} for model \eqref{noisyModel} including a correction factor given by the expected values of the random inputs,
		\begin{equation}\label{ctrl_old}
			u_i(t,\theta) = - \frac{1}{\nu}\left(\left(k_d-\frac{k_o}{N}\right) v_i(t,\theta) + \frac{k_o}{N} \sum_{j=1}^N v_j(t,\theta)\right) - \sum_{k = 1}^Z \mu_k,
		\end{equation}
		for  $\mu_k=\mathbb{E}[\theta_k]$ for $k=1,\ldots,Z$.
	\end{remark}
	\begin{remark}[Averaged control] 
		In the case of a deterministic feedback control, we may consider the expectation of the objective \eqref{eq:u_star}  subject to the noisy model \eqref{noisyModel}
		\begin{equation*}
			\bar u^*(\cdot) = \arg\min_{u(\cdot)} \mathbb{E} \left[  \int_0^{+\infty} \frac{1}{2}( v^\top Q v + \nu u^\top R u)   \ dt \right],
		\end{equation*}
		where we introduce the matrices $Q=R= \frac{1}{N} \Id_N$.  In this case, we have the following deterministic optimal feedback control is deduced
		\begin{equation}\label{eq:u_noiseless}
			\bar u_i(t) = -\frac{1}{\nu} \left( {k}_d \mathbb{E}\left[ v_i(t,\theta) \right] + \frac{{k}_o}{N} \sum_{j \neq i}^N  \mathbb{E}\left[ v_j(t,\theta) \right] \right) - \sum_{k = 1}^Z \mu_k ,
		\end{equation}
		where  $\mu_k=\mathbb{E}[\theta_k]$ for $k=1,\ldots,Z$,  ${k}_d$, ${k}_o$ satisfy equations \eqref{kd_ko}.
		We refer to Appendix \ref{app:noiseless} for detail computations for the synthesis of \eqref{eq:u_noiseless}. 
	\end{remark}

	\section{Robustness  in the $\Hinf$ setting} \label{sec:Hinf}
	In the context of $\Hinf$ theory, controllers are synthesized to achieve stabilization with guaranteed performance. In this section, we exploit the theory of Linear Matrix Inequality (LMI) to show the robustness of the control. The introduction of LMI methods in control has dramatically expanded the types and complexity of the systems we can control. In particular, it is possible to use LMI solvers to synthesize optimal or suboptimal controllers and estimators for multiple classes of state-space systems and without giving a complete list of references, we refer to \cite{mmpeet2020,khalil1996robust,boyd1994linear,duan2013lmis}. 
	
	Consider the linear system \eqref{noisyModel} with control \eqref{control} in the following reformulation
	\begin{align}\label{eq:ctrl_hinf}
		\ddt{v}(t) &= \widehat{A} v(t)+ \widehat{B} \theta,
	\end{align}
	where we consider the random input vector $\theta = (\theta_1,\ldots,\theta_Z)^\top\in \mathbb{R}^{Z\times d}$,
	and the matrices 
	\[
	\widehat{A} = A -\frac{1}{\nu}{K}, \qquad \widehat{B} = \mathds{1}_{N\times Z}.
	\]
	with $\mathds{1}$ a matrix of ones of dimension $N\times Z$. We introduce the frequency transfer function $\hat{G}(s) := (s \Id_N-\widehat{A})^{-1} \widehat{B}$, such that 
	$\hat G \in R\Hinf$. The latter is  the set of proper rational functions with no poles in the closed complex right half-plane, and  the signal norm $\Vert \cdot \Vert_{\Hinf}$ measuring the size of the transfer function in the following  sense 
	\begin{equation}\label{eq:costhinf}
		\Vert \hat{G} \Vert_{\Hinf} = \text{ess}\sup_{\omega \in \mathbb{R}} \bar{\sigma} (\hat{G}(i \omega)),
	\end{equation}
	where for a given matrix $P$, $\bar \sigma (P)$ is the largest singular value of 
	$P$. The general $\Hinf$-optimal control problem consists of finding a stabilizing feedback controller $u = -\frac1\nu K$ which minimizes the cost function \eqref{eq:costhinf} and we refer to Appendix \ref{app:Hinf}   and to \cite{dullerud2013course} for more details. 
	
	However, the direct minimization of the cost $\Vert \hat{G} \Vert_{\Hinf}$ is in general a very hard task, and possibly unfeasible by direct methods.  To  reduce the complexity, a possibility  consists in finding conditions for the stabilizing controller that achieves a norm bound  for a given  threshold $\gamma>0$,
	\begin{equation}\label{eq:Hinf_bound}
		\Vert \hat{G} \Vert_{\Hinf} \leq \gamma.
	\end{equation}
	Hence, robustness of a given control $u= -\frac1\nu K$ is measured in terms of the smallest $\gamma$ satisfying \eqref{eq:Hinf_bound}. 
	In order to provide a quantitative result, we can rely on  the following  result
	\begin{lem}\label{lem:LMI_reduced}
		Given the frequency transfer function $\hat{G}$, associated to \eqref{eq:ctrl_hinf}, a necessary and sufficient condition to guarantee the $\Hinf$ bound \eqref{eq:Hinf_bound} for $\gamma>0$,  is to prove, that there exists a positive  definite square matrix $X$ of order $N$, $X> 0$,  such that the following algebraic Riccati equation holds
		\begin{equation}\label{areu}
			A^\top X + XA -\frac{1}{\nu} {K}^\top X -\frac{1}{\nu} X{K} + \frac{1}{\gamma} XX + \frac{1}{\gamma} \Id_N = \mathcal{O}\left(\frac{1}{N}\right).
		\end{equation}
	\end{lem}
	For detailed proof of this result, we refer to Appendix \ref{app:Hinf}. Note that we are interested in the large particle limit and hence may allow that the previous equation is not exactly zero, but tends to zero at a rate $1/N.$

	%
	%

	\begin{thm} \label{Hinf_thm}
		Consider system \eqref{eq:ctrl_hinf} with structure induced by \eqref{noisyModel}, and consider a square matrix $X$  with the following structure
		\begin{align*}
			(X)_{ij}=
			\begin{cases} 
				x_d,\qquad i=j,\\
				x_o,\qquad i\neq j.
			\end{cases}\qquad 		
		\end{align*}
		Then for $N$ sufficiently large and finite, control $u=-\frac1\nu K$ given by equation \eqref{control} is $\Hinf$-robust for any $\gamma$ and $c_N$ such that $\gamma \geq \frac{1}{c_N}$, $c_N>0$, 	where
		\begin{align}\label{eq:cn_hinf} c_N = \bar{p} + \frac{1}{\nu} (k_d - \frac{k_o}{N}).\end{align}
	\end{thm}
	
	\begin{proof}
		Under the hypothesis of the theorem, \eqref{areu}  reduces to the following system of equations
		\begin{align} \label{xd}
			\mathcal{O}\left(\frac{1}{N}\right) &= \frac{2\bar{p}(1-N)}{N}x_d + \frac{2\bar{p}(N-1)}{N}x_o - \frac{2}{\nu} {k}_d x_d - \frac{2(N-1)}{\nu N}{k}_o x_o \cr
			&\qquad\qquad\qquad\qquad\qquad\qquad\qquad\qquad\qquad+ \frac{1}{\gamma} x_d^2 + \frac{N-1}{\gamma}x_o^2 +\frac{1}{\gamma},
			\\
			\mathcal{O}\left(\frac{1}{N}\right) &= \frac{2\bar{p}(1-N)}{N}x_o + \frac{2\bar{p}}{N} x_d + \frac{2\bar{p}(N-2)}{N}x_o - \frac{2}{\nu} {k}_d x_o - \frac{2}{\nu N}{k}_o x_d 
			\cr
			&\qquad\qquad\qquad\qquad\qquad\qquad\qquad- \frac{2(N-2)}{\nu N} {k}_o x_o + \frac{2}{\gamma} x_d x_o + \frac{N-2}{\gamma}x_o^2. \label{xo}
		\end{align}
		We scale the off-diagonal elements $x_o$ of $X$ according to 
		\begin{equation*}
			\quad \tilde{x}_o = \sqrt{N} x_o,
		\end{equation*}
		which, as we will see later, is also the consistent scaling for a mean-field description of $X.$ 
		\par 
		Then, the previous system reads
		\begin{align} \label{AREdiscrete}
			\begin{split}
				\mathcal{O}\left(\frac{1}{N}\right) &= \frac{2\bar{p}(1-N)}{N}x_d + \frac{2\bar{p}(N-1)}{N\sqrt{N}}\tilde{x}_o - \frac{2}{\nu} k_d x_d - \frac{2(N-1)}{\nu N \sqrt{N}}k_o \tilde{x}_o + \frac{1}{\gamma} x_d^2 + \frac{N-1}{\gamma N}\tilde{x}_o^2 +\frac{1}{\gamma},
				\\
				\mathcal{O}\left(\frac{1}{N}\right) &= \frac{2\bar{p}(1-N)}{N\sqrt{N}}\tilde{x}_o + \frac{2\bar{p}}{N} x_d + \frac{2\bar{p}(N-2)}{N\sqrt{N}}\tilde{x}_o - \frac{2}{\nu\sqrt{N}} k_d \tilde{x}_o - \frac{2}{\nu N}k_o x_d - \frac{2(N-2)}{\nu N \sqrt{N}} k_o \tilde{x}_o \cr&\qquad\qquad+ \frac{2}{\gamma \sqrt{N}} x_d \tilde{x}_o + \frac{N-2}{\gamma N}\tilde{x}_o^2.
			\end{split}
		\end{align}
		From these two equations of system \eqref{AREdiscrete}, and setting
		\begin{equation}\label{eq:c}
			c = \bar{p} + \frac{k_d}{\nu}, 	\quad	\alpha = \bar{p} - \frac{k_o}{\nu} , \quad \beta = \frac{k_d + k_o}{\nu}.
		\end{equation}
		we obtain two second order equations for $x_d$ and $\tilde{x}_o$
		\begin{equation}\label{eq:system_xdxo} 
			\begin{split}
				0 &= x_d^2 -2\gamma c x_d + \tilde{x}_o^2 + \frac{2\gamma \alpha }{\sqrt{N}} \tilde{x}_o + 1 + \mathcal{O} \left( \frac{1}{N}\right),  \\
				0 &= \tilde{x}_o^2 - \frac{2\gamma}{\sqrt{N}} \Bigl( \beta - \frac{x_d}{\gamma} \Bigr) \tilde{x}_o + \mathcal{O}\left( \frac{1}{N}\right).
			\end{split}
		\end{equation}
		For $N$ sufficiently large, their solutions is given by 
		\begin{align*}
			x_d^{\pm} &= \gamma c \pm \sqrt{\gamma^2c^2 - 1 - \tilde{x}_o^2 - \frac{2\gamma \alpha}{\sqrt{N}}\tilde{x}_o}, \quad 
			\tilde{x}_o^{-} = 0, \quad	\tilde{x}_o^{+} =  \frac{2\gamma}{\sqrt{N}}  \Bigl( \beta - \frac{x_d}{\gamma} \Bigr).
		\end{align*}
		Hence, we write the matrix $X$ as
		\begin{equation*}
			X = \frac{\tilde{x}_o^{\pm} }{\sqrt{N}} \ \mathds{1}_{N} + (x_d^{\pm}  - \frac{\tilde{x}_o^{\pm} }{\sqrt{N}}) \ {Id}_{N},
		\end{equation*}
		and the eigenvalues of $X$ are
		\begin{equation*}
			\lambda_i = \lambda = x_d^{\pm}  - \frac{\tilde{x}_o^{\pm} }{\sqrt{N}} \quad \text{for} \quad i = 1,\dots, N-1 \quad \text{and} \quad \lambda_N = x_d^{\pm}  + (N-1) \frac{\tilde{x}_o^{\pm} }{\sqrt{N}}.
		\end{equation*}

		If we subtract equations in \eqref{AREdiscrete} we find the following second order equation in the variable $\lambda = (x_d - \frac{\tilde{x}_o}{\sqrt{N}})$
		
		\begin{equation*}
			\lambda^2 - 2\gamma \Bigl(\bar{p} + \frac{1}{\nu} ({k}_d - \frac{k_o}{N})\Bigr) \lambda + 1 = 0,
		\end{equation*}
		
		with its solutions
		$	\lambda^{\pm} = x_d - \frac{\tilde{x}_o}{\sqrt{N}} =  \gamma c_N \pm \sqrt{ \gamma^2 c_N^2 - 1},
		$
		and for 
		\begin{equation}\label{c_N}
			c_N = \bar{p} + \frac{1}{\nu} \left(k_d - \frac{k_o}{N}\right).
		\end{equation}
		Therefore,  we have	$ \lambda_i = \lambda^{\pm} \quad \text{for} \quad i = 1,\dots, N-1$ and $\lambda_N = \lambda^{\pm} + N \frac{\tilde{x}_o}{\sqrt{N}}.$ 
		Hence, there exists a matrix $X$  satisfying \eqref{areu}, for  $$x_o = x_o^- = 0.$$ In this case, the eigenvalues are $\lambda_i = \lambda_N= \lambda^{\pm}$.
		In order to ensure the  positivity of the eigenvalue $\lambda^{\pm} = \gamma c_N \pm \sqrt{\gamma^2 c_N^2-1}$ needs to be non--negative. This can be expressed in terms of the choice of the parameters  
		$\gamma$ and $c_N:$	 First of all, we have  $\gamma \geq \frac{1}{c_N}$ to ensure the existence of the square root. In addition, provided that $c_N>0$ we obtain 
		\begin{align*}
			& \lambda^+> 0  \mbox{ and } \lambda^- >0.
		\end{align*}
		This finishes the proof. 
	\end{proof}
	
	\begin{remark} We observe that Theorem \ref{Hinf_thm} quantifies the robustness of the feedback control with a lower bound on the parameters of the model. In particular, we can achieve minimal value of $\gamma$ for  large values of $c_N$ in \eqref{eq:cn_hinf}, for example if we fix the values $\bar p, k_d,k_o$ and $N$  for decreasing values of the penalization $\nu$ the control is more robust.
	\end{remark}

	\subsection{Mean-field estimates for $\Hinf$ control}
Large system of interacting agents can be efficiently represented at different scales to reduce the complexity of the microscopic description. Of particular interest are models able to describe the evolution of a density of agents and its moments  \cite{carrillo2014derivation,carrillo2010particle,MR2425606}.

In this section, we analyse robustness of controls in the case of a large number of agents, i.e.  $N\gg 1$, by means of the {\em mean-field limit} of the interacting system.   Hence, we consider the density distribution of agents $f=f(t,v,\theta)$ to describe the collective behaviour of the ensemble of agents. The empirical joint probability distribution of agents for the system \eqref{noisyModel}, 
	is given by
	\begin{equation*}
		f^N(t,v,\theta) = \frac{1}{N} \sum_{i=1}^N \delta (v-v_i(t,\theta)),
	\end{equation*}
	where $\delta(\cdot)$ is a Dirac measure over the trajectories $v_i(t,\theta)$ dependent on the stochastic variable $\theta = (\theta_1,\ldots,\theta_Z)$. 
	
	Hence,  assuming enough regularity assuming that agents remain in a fixed compact domain for all $N$ and in the whole time interval $[0,T]$,  the mean-field limit of dynamics \eqref{noisyModel} is obtained formally as
		\begin{equation*}
		\partial_t f(t,v,\theta) = - \nabla_v\cdot \left( f(t,v,\theta) \left(\left( \bar{p}-\dfrac{k_o}{\nu}\right) m_1[f](t,\theta) - \left(  \bar{p}+\frac{k_d}{\nu}\right)v + \sum_{k=1}^Z\theta_k \right)\right),
	\end{equation*}
	with initial data $ f(0,v,\theta) = f^0(v,\theta)$. 
	
	The latter is obtained as the limit of $f^N(0,v,\theta)$ in the Wasserstein distance given a sequence of initial agents, \cite{carrillo2009double,carrillo2014derivation}.
	The quantity  $m_1[f]$ denotes the first moment of $f$ with respect to $v$  
	\begin{align*}
		m_1[f](t,\theta) = \int_{\R^d} v f(t,v,\theta) dv.
	\end{align*}
	For the many-particle limit, we recover a mean-field estimate of the $\Hinf$ condition similarly as in \ref{Hinf_thm}.
	Indeed, for $N\rightarrow\infty$ the nonlinear system \eqref{AREdiscrete} yields 
	\begin{align*}
		0 &= -2\bar{p}x_d - \frac{2}{\nu} k_d x_d + \frac{1}{\gamma} x_d^2 + \frac{1}{\gamma}x_o^2 +\frac{1}{\gamma}, \; 	0 = \frac{1}{\gamma}x_o^2.
	\end{align*}
	Hence, for any fixed finite $N$ the matrix $X$ is diagonal with the entry 
	\begin{equation*}
		x_d^{\pm} = \gamma \Bigl( \bar{p}+\frac{k_d}{\nu}\Bigr) \pm \sqrt{\gamma^2\Bigl( \bar{p}+\frac{k_d}{\nu}\Bigr)^2-1} + O(\frac{1}N). 
	\end{equation*}
	To ensure that $X$ is positive definite, we only have to assume that $\gamma\geq\dfrac{1}c,$ 
	where  $\gamma$ is the bound of the $\Hinf$ norm, and $c = \bar{p}+k_d/\nu + O(\frac{1}N)$ corresponds to the value defined by equation
	\eqref{c_N}. This shows that for any $N$ there exists a positive definite matrix that guarantees robust stabilization. Note that the condition for any $N$ is the limit
	of the finite-dimensional conditions of the previous Lemma \ref{lem:LMI_reduced} and Theorem \ref{Hinf_thm}. 
	
	\begin{remark} For explicit values of the Riccati coefficients we can characterize the previous estimates more precisely. In particular,  for  $N\rightarrow\infty$  system
		\eqref{kd_ko} reduces to
		\begin{align*}
			0 =  \frac{k_d^2}{\nu} +\left(2\bar{p}+r\right)k_d - 1,\qquad
			0 =   \frac{k_o^2}{\nu} + \left(\frac{2}{\nu}k_d +r\right)k_o - 2\bar{p}k_d,
		\end{align*}
		with solutions
		\begin{equation*}
			k_d^{\pm} = -\nu \left(\bar{p} + \frac{r}{2} \right)\pm \nu\sqrt{\left(\bar{p} + \frac{r}{2} \right)^2+\frac{1}{\nu}}, \quad \quad k_o^{\pm} = -\left(k_d+ \frac{\nu r}{2} \right) \pm \sqrt{\left(k_d+ \frac{\nu r}{2} \right) ^2+2 \nu \bar{p} k_d}.
		\end{equation*}
		In this particular case, the condition of Theorem \ref{Hinf_thm} becomes 
		\begin{equation}\label{eq:gammabound_r}
			\gamma \geq \frac{\sqrt{\nu}}{{\sqrt{\left(\bar p +\frac{r}{2}\right)^2\nu+1}}-\frac{r\sqrt{\nu}}{2}}.
		\end{equation}
		
		In	Figure \ref{fig:gamma} we depict the lower bound of $\gamma$ for $r=0$ 	for different values of $\nu$ and $\bar{p}$. As expected,  smaller values of $\gamma$, hence more robustness, is obtained if the penalization factor  $\nu$  is small or when $\bar{p}$ is large. The latter corresponds to a stronger attraction between agents.
		\begin{figure}[H]
			\begin{center}
				\includegraphics[width=0.5\linewidth]{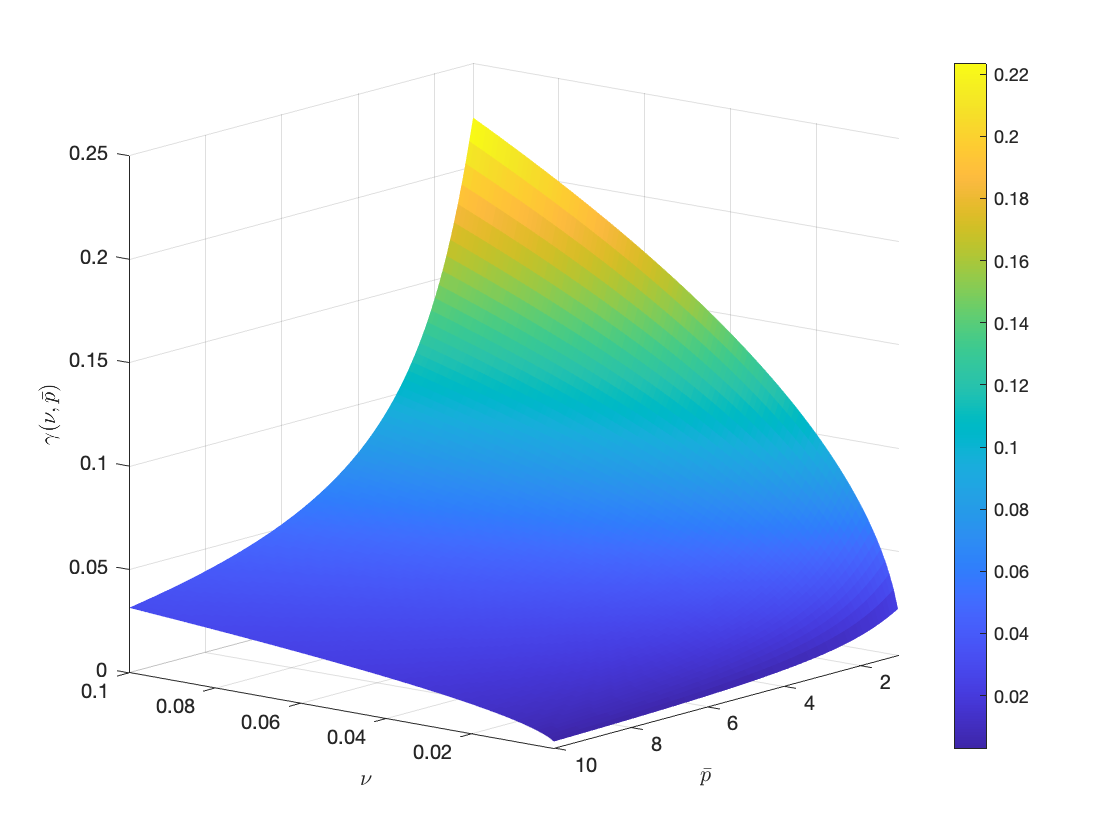}
			\end{center}
			\caption{Numerical computation of the lower bound value of $\gamma$ in \eqref{eq:gammabound_r} as  function of $\nu$ and $\bar{p}$ when $r=0$, i.e. 
				$\gamma(\nu, \bar{p}) = \sqrt{{\nu}/(\bar{p}^2 \nu +1)}.$}\label{fig:gamma}
		\end{figure}
	\end{remark}

	\section{Numerical approximation of the uncertain dynamics}\label{sec:numerics_noise}
	
	In this section, we present numerical tests based on linear microscopic and mean-field equations in presence of uncertainties. In particular, we give numerical evidence of the robustness of the feedback control  \eqref{control} and illustrate a comparison with the averaged control  \eqref{eq:u_KS_comp}. For the numerical approximation of the random space, we employ the Stochastic Galerkin (SG) method belonging to the class of generalized polynomial chaos (gPC) (\cite{le2010spectral,xiu2010numerical}). In the mean-field setting, the evolution of the density distribution is approximated with a Monte Carlo (MC) method, in a similar spirt of particle based gPC techniques developed in \cite{CiCP-25-508}. 
	

	\subsection{SG approximation for  robust constrained interacting agent systems}
	We approximate the dynamics using a stochastic Galerkin approach applied to the interacting particle system with uncertainties \cite{article_Albi_2015,CiCP-25-508}. Polynomial chaos expansion provides a way to represent a random variable with finite variance as a function of an $M$-dimensional random vector using a polynomial basis that is orthogonal to the distribution of this random vector. 
	Depending on the distribution, different expansion types are distinguished, as shown in Table \ref{tab:gPC}.
	\begin{table}
		\centering
		\begin{tabular}{c|c|c}		
			Probability law of $\theta$ & Expansion polynomials & Support\\
			\hline
			Gaussian & Hermite & $(-\infty, +\infty)$\\
			Uniform & Legendre & $[a,b]$ \\
			Beta & Jacobi & $[a,b]$  \\
			Gamma & Laguerre & $[0, +\infty)$\ \\
			Poisson & Charlier & $\mathbb{N}$ \\
			\hline
		\end{tabular}
		\caption {The different choices for the polynomial expansions.}\label{tab:gPC}
		\vskip-5mm
	\end{table}
	We recall first some basic notions on gPC approximation techniques and for the sake of simplicity we consider a one-dimensional setting for the dynamical state $v_i,$ i.e., $d=1$.
	Let $(\Omega,\mathcal{F}, P)$ be a probability space where $\Omega$ is an abstract sample space, $\mathcal{F}$ a $\sigma-$algebra of subsets of $\Omega$ and $P$ a probability measure on $\mathcal{F}$. Let us define a random variable
	\begin{equation}\label{eq:theta_noise}
		\theta : (\Omega, \mathcal{F}) \rightarrow (I_\Theta, \mathcal{B}(\mathbb{R}^Z))
	\end{equation}
	where $I_\Theta \in \mathbb{R}^Z$ is the range of $\theta$ and $\mathcal{B}(\mathbb{R}^Z)$ is the Borel $\sigma$-algebra of subsets of $\mathbb{R}^Z$, we recall that $Z$ is the dimension of the random input $\theta=(\theta_1,\ldots,\theta_Z)$ and where we assume that each component is independent.
	We consider the linear spaces generated by orthogonal polynomials of $\theta_j$ with degree up to $M$: $\lbrace \Phi_{k_j}^{(j)} (\theta)\rbrace_{k_j=0}^M$, with $j = 1,\dots, Z$. Assuming that the probability law for the function $v_i(t,\theta)$ has a finite second order moment, the complete polynomial chaos
	expansion of $v_i$ is given by
	\[
	v_i (t,\theta) = \sum_{k_1,\ldots,k_Z \in \mathbb{N}} \hat{v}_{i,k_1\dots k_Z} (t) \prod_{j=1}^Z\Phi_{k_j}^{(j)}(\theta_{j}),
	\]
	where the coefficients $\hat{v}_{i,k_1\dots k_Z} (t)$ are defined as 
	\[
	\hat{v}_{i,k_1\dots k_Z}(t) = \mathbb{E}_\theta \left[ v_i(t,\theta) \prod_{j=1}^Z\Phi^{(j)}_{k_j} (\theta_j)\right]
	\]
	where the expectation operator $\mathbb{E}_{\theta}$ is computed with respect to the joint distribution $\rho=\rho_1\otimes\ldots\otimes \rho_Z$, and where $\{\Phi^{(j)}_{k}(\theta_j)\}_k$ is a set of polynomials which constitute the optimal basis with respect to the known distribution $\rho(\theta_j)$ of the random variable $\theta_j$, such that 
	\begin{align*}
		\mathbb{E}_{\theta_j} \left[ \Phi^{(j)}_{k}(\theta_1) \Phi^{(j)}_h(\theta_j) \right] &= \mathbb{E}_{\theta_j} \left[ \Phi^{(j)}_h(\theta_j)^2\right] \delta_{hk},
	\end{align*}
	with $\delta_{hk}$ the Kronecker delta.
	From the numerical point of view, we may have an exponential order of convergence for the SG series expansion, unlike Monte Carlo techniques for which
	the order is $\mathcal{O}(1/\sqrt{M})$ where $M$ is the number of samples. Considering the noisy model \ref{noisyModel} with control $u_i(t)$ in \ref{control}, we have
	\begin{equation}\label{eq:dyn_old_ctrl}
		\dot{v}_i (t,\theta) = \frac{\bar{p}}{N} \sum_{j = 1}^N \Bigl(v_j (t,\theta) - v_i (t,\theta)\Bigr) - \frac{1}{N \nu} \sum_{j = 1}^N \Bigl(k_o(t)v_j (t,\theta) + k_d(t) v_i (t,\theta)\Bigr) + \sum_{j = 1}^Z \theta_j.
	\end{equation}
	We apply the SG decomposition to the solution of the differential equation $v_i(t,\theta)$ in \eqref{eq:dyn_old_ctrl} and to the stochastic variable $\theta_k$, and for $i=1,\dots, N, \ l=1,\dots,Z$, we have
	\begin{equation}
		\begin{split}
			v_i^M (t,\theta) &=\sum_{k_1,\ldots,k_Z =0}^M \hat{v}_{i,k_1\dots k_Z} (t) \prod_{j=1}^Z\Phi_{k_j}^{(j)}(\theta_{j}), 
			\\
			\theta_l^M (\theta)& = \sum_{k_1,\ldots,k_Z =0}^M \hat{\theta}_{l,k_1\dots k_Z} (t) \prod_{j=1}^Z\Phi_{k_j}^{(j)}(\theta_{j}), 
		\end{split}
	\end{equation}
	with
	\[
	\hat{\theta}_{l,k_1\dots k_Z} = \mathbb{E}_\theta \left[ \theta_l \prod_{j=1}^Z\Phi^{(j)}_{k_j} (\theta_j)\right]
	= \mathbb{E}_{\theta_l} \left[ \theta_l \Phi^{(l)}_{k_l} (\theta_l)\right]
	\prod_{j=1, j\neq l}^Z \mathbb{E}_{\theta_j} \left[ \Phi^{(j)}_{k_j} (\theta_j)\right].
	\]
	Then we obtain the following polynomial chaos expansion
	\begin{equation}\label{eq:gPC1}
		\begin{split}
			& \frac{d}{dt} \sum_{k_1,\ldots,k_Z =0}^M  \hat{v}_{i,k_1\dots k_Z} \prod_{j=1}^Z\Phi_{k_j}^{(j)}(\theta_{j}) =
			\\
			&= \frac{1}{N} \sum_{h = 1}^N  \sum_{k_1,\ldots,k_Z =0}^M  \left[ \left( \bar{p} - \frac{k_o}{\nu} \right) \hat{v}_{h,k_1\dots k_Z}- \left( \bar{p} + \frac{k_d}{\nu} \right) \hat{v}_{i,k_1\dots k_Z} \right]\prod_{j=1}^Z\Phi_{k_j}^{(j)}(\theta_{j}) 
			\\
			& \quad+ \sum_{l=1}^Z \sum_{k_1,\ldots,k_Z =0}^M \hat{\theta}_{l,k_1\dots k_Z} (t) \prod_{j=1}^Z\Phi_{k_j}^{(j)}(\theta_{j}).
		\end{split}
	\end{equation}
	Multiplying by $\prod_{j=1}^Z\Phi_{k_j}^{(j)}(\theta_{j}) $ and integrating with respect to the distribution $\rho(\theta)$, we end up with
	\begin{align}
		\frac{d}{dt} \hat{v}_{i,k_1\dots k_Z} = \frac{1}{N} &\sum_{h = 1}^N  \left[ \left( \bar{p} - \frac{k_o}{\nu} \right) \hat{v}_{h,k_1\dots k_Z} - \left( \bar{p} + \frac{k_d}{\nu} \right) \hat{v}_{i,k_1\dots k_Z} \right]
		\\
		&+  \sum_{l= 1}^Z \mathbb{E}_{\theta_l} \left[ \theta_l \Phi^{(l)}_{k_l} (\theta_l)\right]
		\prod_{j=1, j\neq l}^Z \mathbb{E}_{\theta_l} \left[ \Phi^{(j)}_{k_j} (\theta_j)\right].
	\end{align}
	For the numerical tests, we approximate the integrals using quadrature rules.
	\begin{remark}
		For  model \ref{eq:u_noiseless}, where the control is averaged with respect to the random sources,
		\begin{equation}
			u_i = -\frac{1}{\nu} \left( {k}_d \mathbb{E}_\theta\left[ v_i \right] + \frac{{k}_o}{N} \sum_{h \neq i}^N  \mathbb{E}_\theta\left[ v_h\right] + {s}\sum_{j = 1}^Z \mathbb{E}_\theta\left[ \theta_j \right] \right),
		\end{equation}
		the SG approximation is given by 
		\begin{equation} \label{eq:gPC2}
			\begin{split}
				\frac{d}{dt} \hat{v}_{i,k_1\dots k_Z}  &= \frac{\bar{p}}{N} \sum_{h = 1}^N  \left( \hat{v}_{h,k_1\dots k_Z} -  \hat{v}_{i,k_1\dots k_Z} \right) 
				\\ 
				- &\frac{ 	\prod_{j=1}^Z \mathbb{E}_{\theta_j} \left[ \Phi^{(j)}_{k_j} (\theta_j)\right] }{	\prod_{j=1}^Z \mathbb{E}_{\theta_j} \left[ \left( \Phi^{(j)}_{k_j} (\theta_j)\right)^2\right]} \left( {k}_d \hat{v}_{i,00\dots0} + \frac{{k}_o}{N} \sum_{h \neq i}^N {k}_o \hat{v}_{h,00\dots 0} + s \sum_{j = 1}^Z \mu_j \right)
				\\
				+ & \sum_{l= 1}^Z \mathbb{E}_{\theta_l} \left[ \theta_l \Phi^{(l)}_{k_l} (\theta_l)\right]
				\prod_{j=1, j\neq l}^Z \mathbb{E}_{\theta_j} \left[ \Phi^{(j)}_{k_j} (\theta_j)\right].
			\end{split}
		\end{equation}
		We recover the mean and the variance of the random variable $v(\theta)$ as
		\begin{align*}
			\mathbb{E}_\theta[v_i(\theta)] &= \int_{\mathbb{R}^2} v_i(\theta) d\rho = \hat{v}_{i,00\dots 0} , \\
			\mathbb{V}_\theta[v_i(\theta)] &= \int_{\mathbb{R}^2} ( v_i(\theta) - \hat{v}_{i,00\dots 0} )^2 d\rho = \sum_{k_1,\ldots,k_Z =0}^M  \hat{v}_{i,k_1\dots k_Z}^2 - \hat{v}_{i,00\dots 0}^2 .
		\end{align*}
	\end{remark}

	\subsection{Numerical tests}
	In this section we present different numerical tests on microscopic and mean-field dynamics, to compare the robustness of controls described in sections \ref{sec:control}. We analyze one- and two-dimensional dynamics, 
	for every test we consider the attractive case with $\bar{p}=1$. The initial distribution of agents $v_0$ is chosen such that consensus to the target $\bar v  = 0$  would not be reached without control action.
	We implement the SG approximations in \eqref{eq:gPC1} and \eqref{eq:gPC2} and we perform the time integration until the final time $T=1$ of the resulting system through a 4th order Runge-Kutta method.
	
	We are taking into account a dynamics with $Z= 2$ additive uncertanties,  $\theta_1$ a random variable with gaussian density distribution $\rho_1 \sim \mathcal{N}(\mu, \sigma^2)$, and  $\theta_2$ with uniform density distribution $\rho_2 \sim \mathcal{U}(a, b)$.
	This assumption of normal and uniform distributions for the stochastic parameter corresponds to the case of Hermite and Legendre polynomial chaos expansions, respectively, as shown in Table \ref{tab:gPC}.
	For every test, we have $M=10$ terms of the SG decomposition.
	
	\subsubsection{Test 1: one-dimensional consensus dynamics}
	
	In the one-dimensional microscopic case we take $N=100$ agents, and a uniform initial distribution of agents, $v_0\sim \mathcal{U}(10,20)$.  Figure \ref{fig:1D} shows means, as continuous and dashed lines, and confidence regions of the two noisy dynamics for different distributions $\rho_1, \ \rho_2$.
	The shaded region is computed as the region between the values 
	\begin{equation*}
		\frac{1}{N} \sum_{i=1}^N \left( \mathbb{E}_\theta[v_i(\theta)] \right) \ \mp \ \max_{i=1,\dots,N} \left( \sqrt{\mathbb{V}_\theta[v_i(\theta)]} \right).
	\end{equation*}
	\begin{figure}[H]
		\begin{center}
			\includegraphics[width=0.45\linewidth]{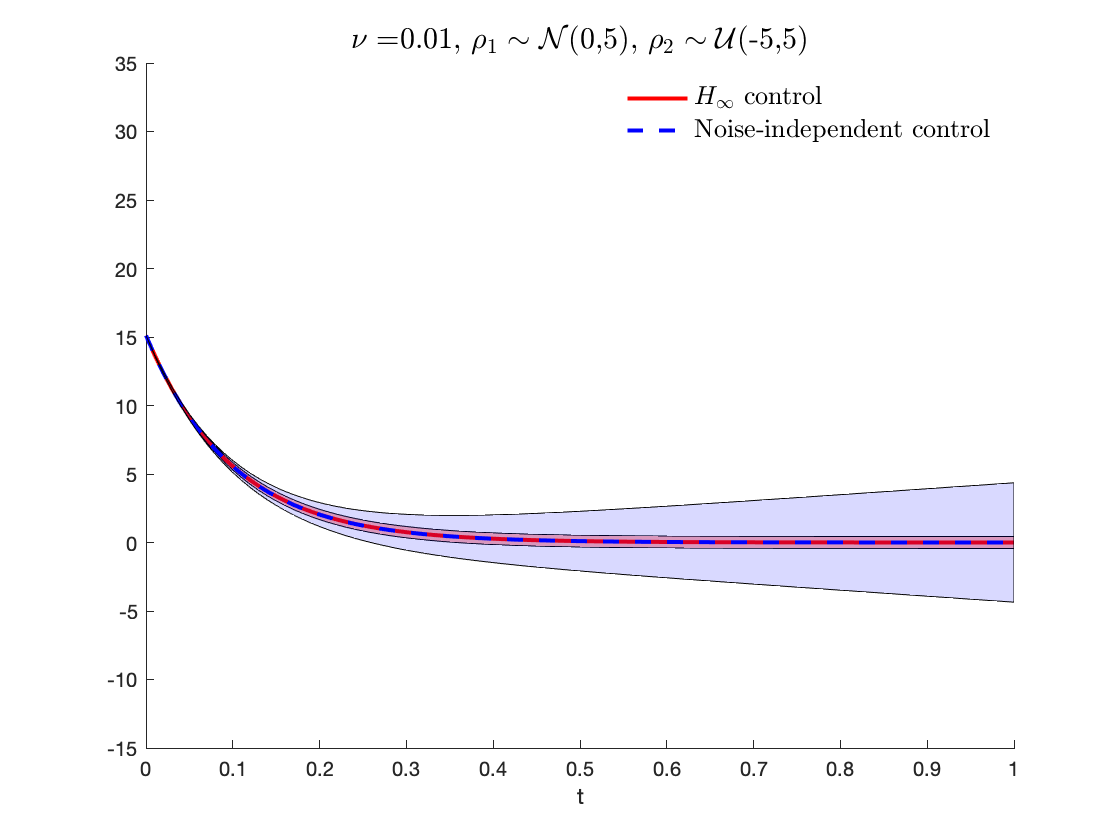}
			\includegraphics[width=0.45\linewidth]{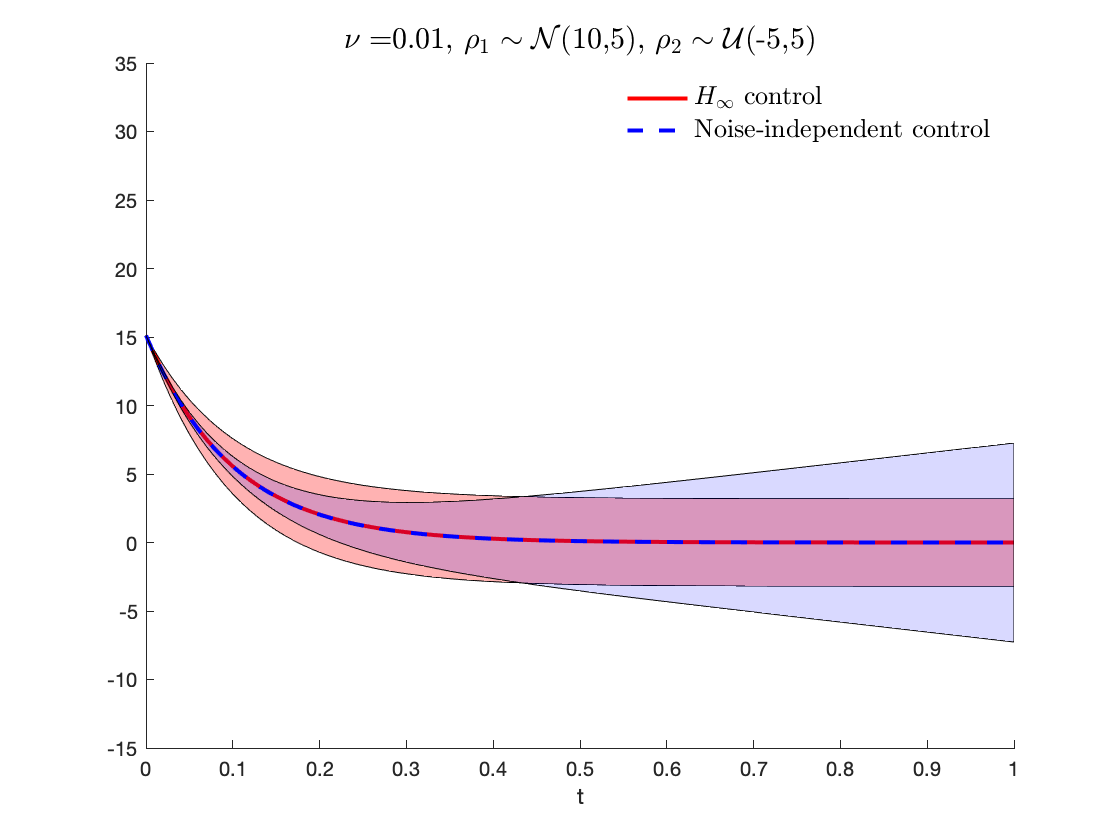}
		\end{center}
		\caption{{\em Test 1.} Comparison between the two controls in \eqref{control} and \eqref{eq:u_noiseless} applied to the multiagent system with uncertainties, in terms of dynamics mean and variance for different $\rho_1, \rho_2$.}
		\label{fig:1D}
	\end{figure}
	
	Numerical results show that both introduced controls are capable to drive the agents to the desired state even in the case of a dynamic dependent on random inputs.
	Moreover, we can observe that, with the $\Hinf$ control, the variance of the uncertain dynamics is stabilized over time, while in the case of a averaged control the variance keeps growing. This is because the averaged control has information only on the mean value of the state and uncertainty, while the $\Hinf$ feedback control directly depends on the state, and as a consequence on the randomness of the dynamics.
	This is also expected given the robustness estimate on the feedback control.

	\subsubsection{Test 2: two-dimensional consensus dynamics}
	
	In the two-dimensional microscopic case, we take $N=100$ agents, and an initial configuration is uniformly distributed on a 2D disc, as shown in Figure \ref{fig:init_distr}.
	2D means and confidence regions of the two noisy dynamics can be seen in Figure \ref{fig:2D},  for different values of the penalization factor $\nu$ and different distributions $\rho_1, \ \rho_2$.
	\begin{figure}[H]
		\begin{center}
			\includegraphics[width=0.45\linewidth]{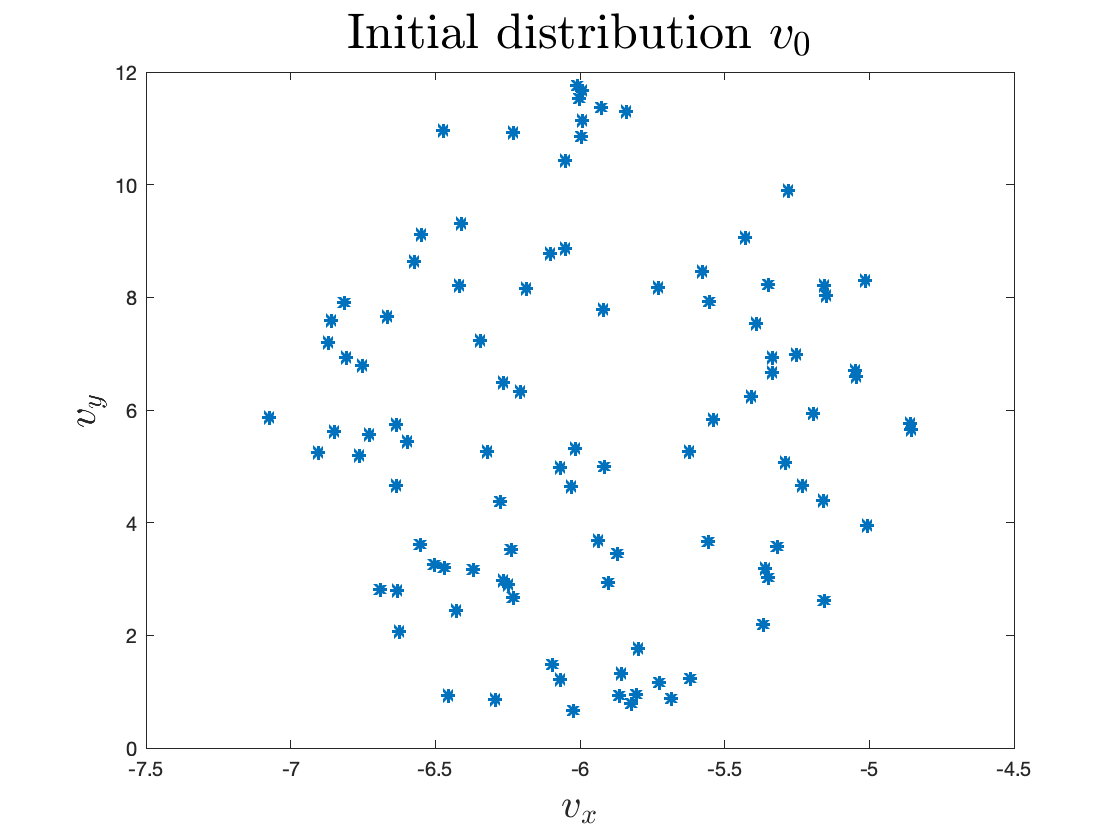}
		\end{center}
		\caption{Test 2: Initial distribution of the agents in the two-dimensional setting, $v_0\in\R^{2N}$.}
		\label{fig:init_distr}
	\end{figure}
	We recall that, for the control $u$ in Eq. \eqref{control}, the size of the transfer function $\hat{G}$ related to the state-space system \eqref{eq:ctrl_hinf} in terms of the $\Hinf$ signal norm is 
	\begin{equation*}
		\Vert \hat{G} \Vert_{\Hinf} \leq \gamma.
	\end{equation*}
	From Theorem \ref{Hinf_thm} we know that the $\Hinf$ control $u$ is robust with a constant $\gamma > \frac{1}{c_N}$, $c_N>0$. We compute the value $c_N$ for the two cases in Figure \ref{fig:2D}, and we have $c_N = 14.29$ for a penalization factor $\nu =0.01$, while $c_N = 4.55$
	for $\nu = 0.1$. As expected, we observe smaller regions for a smaller value of  $\nu$, interpreted as the control cost.
	\begin{figure}[H]
		\begin{center}
			\includegraphics[width=0.45\linewidth]{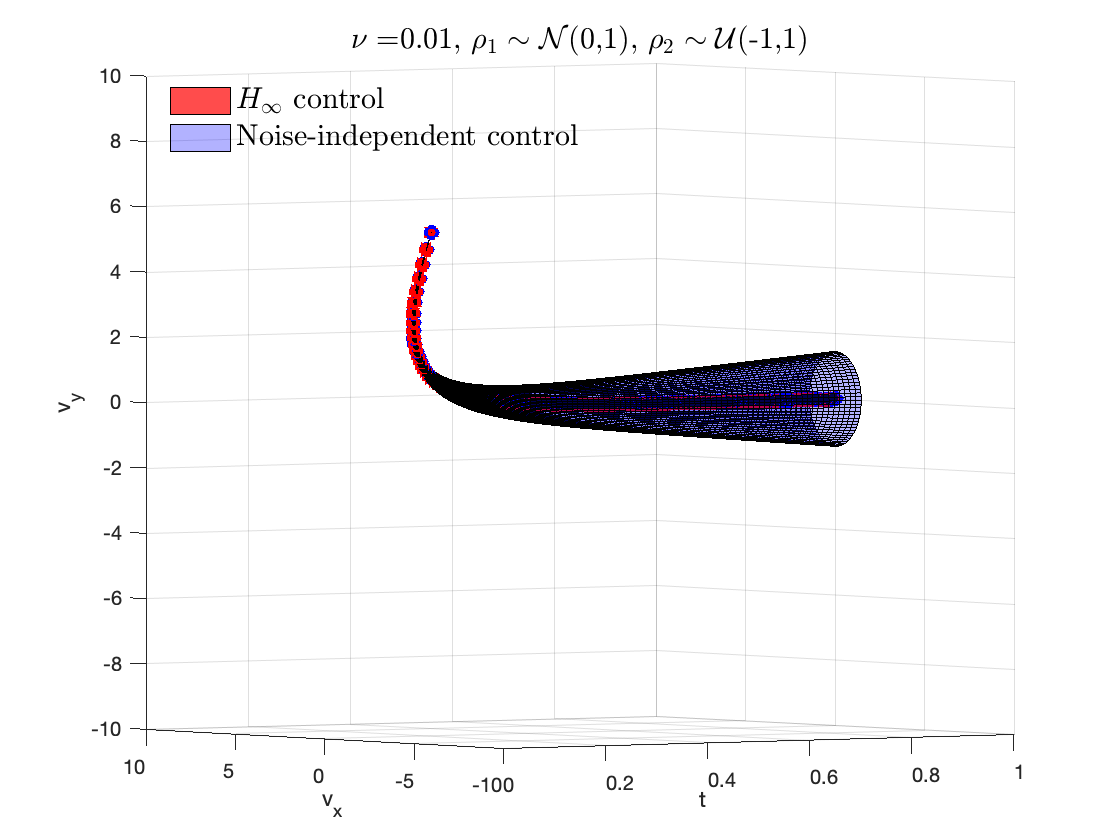}
			\includegraphics[width=0.45\linewidth]{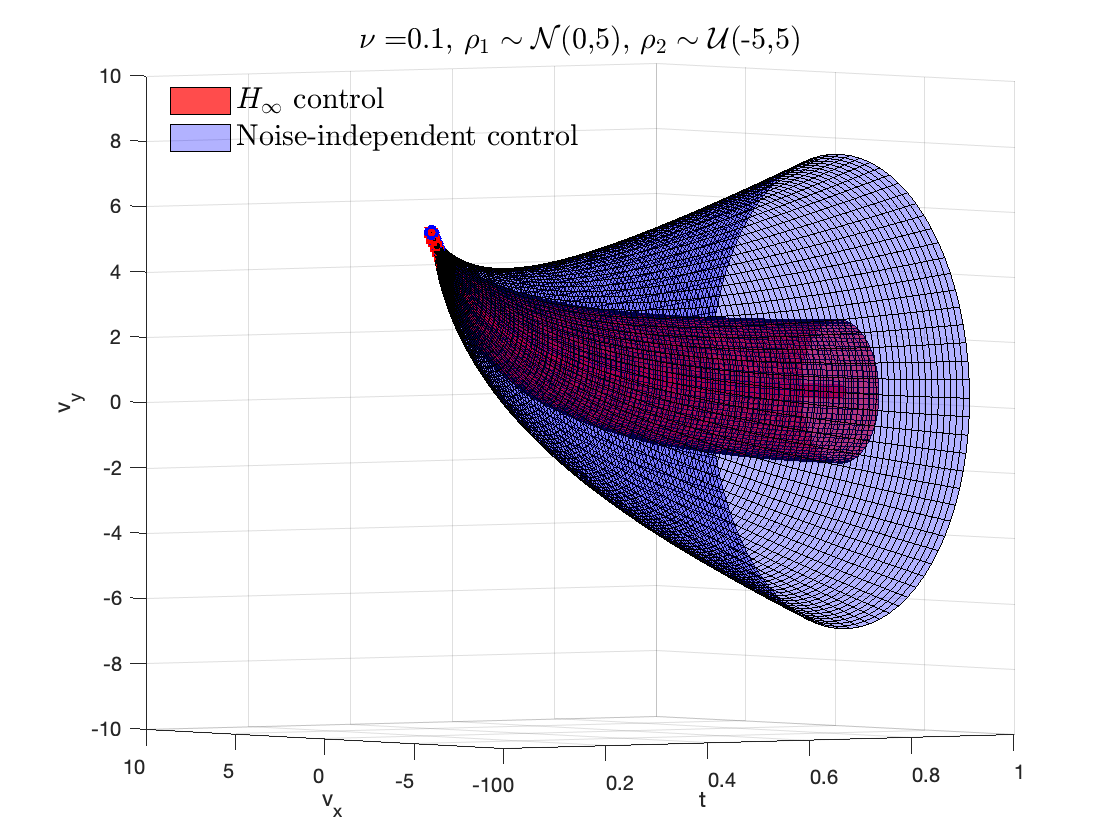}
		\end{center}
		\caption{{\em Test 2.} Two-dimensional case. Comparison between the two controls in \eqref{control} and \eqref{eq:u_noiseless} applied to the uncertain model, in terms of dynamics mean and variance for different values of $\nu, \rho_1, \rho_2$.}
		\label{fig:2D}
	\end{figure}
	
	\subsubsection{Test 3: mean-field consensus dynamics}
	
	In the mean-field limit, the Monte Carlo (MC) method is employed for the approximation of the distribution function $f(t,v,\theta)$ in the phase space whereas the random space at the particle level is approximated through the SG technique.
	
	Considering this MC-SG scheme, we work on an agent system using  Monte Carlo sampling with $N_s = 10^4$ agents, then we consider the SG scheme at the microscopic level. The probability density $f(t,v,\theta)$ is then  reconstructed as the histogram of $v(t,\theta)$. The reconstruction step of the mean density has been done with $50$ bins. The mean and the variance of the statistical quantity are computed as follows

	\begin{equation}\label{eq:quad}
		\begin{split}
			\mathbb{E}_\theta[f(t,v,\theta)] &= \int \int f(t,v,\theta) d\rho_1(\theta_1) d\rho_2(\theta_2) 
			= \sum_{l,h=1}^L f(t,v,\theta^{lh}) \rho_1(\theta_1^l) \rho_2(\theta_2^h) \omega_1^l \omega_2^h,
			\\
			\mathbb{V}_\theta[f(t,v,\theta)] &= \int \int f(t,v,\theta)^{2} d\rho_1(\theta_1) d\rho_2(\theta_2)  - \left( \mathbb{E}_\theta[f(t,v,\theta)] \right)^2
			\\
			&= \sum_{l,h=1}^L f(t,v,\theta^{lh})^{2}  \rho_1(\theta_1^l) \rho_2(\theta_2^h) \omega_1^l \omega_2^h - \left( \mathbb{E}_\theta[f(t,v,\theta)] \right)^2,
		\end{split}
	\end{equation}
	where for $l,h = 1,\dots, L$ and $\hat{v}_{k,j} \in \mathbb{R}^{N_s}$, $f(t,v,\theta^{lh}) $ is reconstructed as the histogram of the data 
	$v(\theta^{lh}) = \sum_{k,j=0}^M \hat{v}_{k,j} \Phi_k(\theta_1^l) \Psi_j(\theta_2^h).$
	
	\begin{figure}[H]
			\begin{subfigure}[b]{0.5\textwidth}
				\centering
				\includegraphics[width=\textwidth]{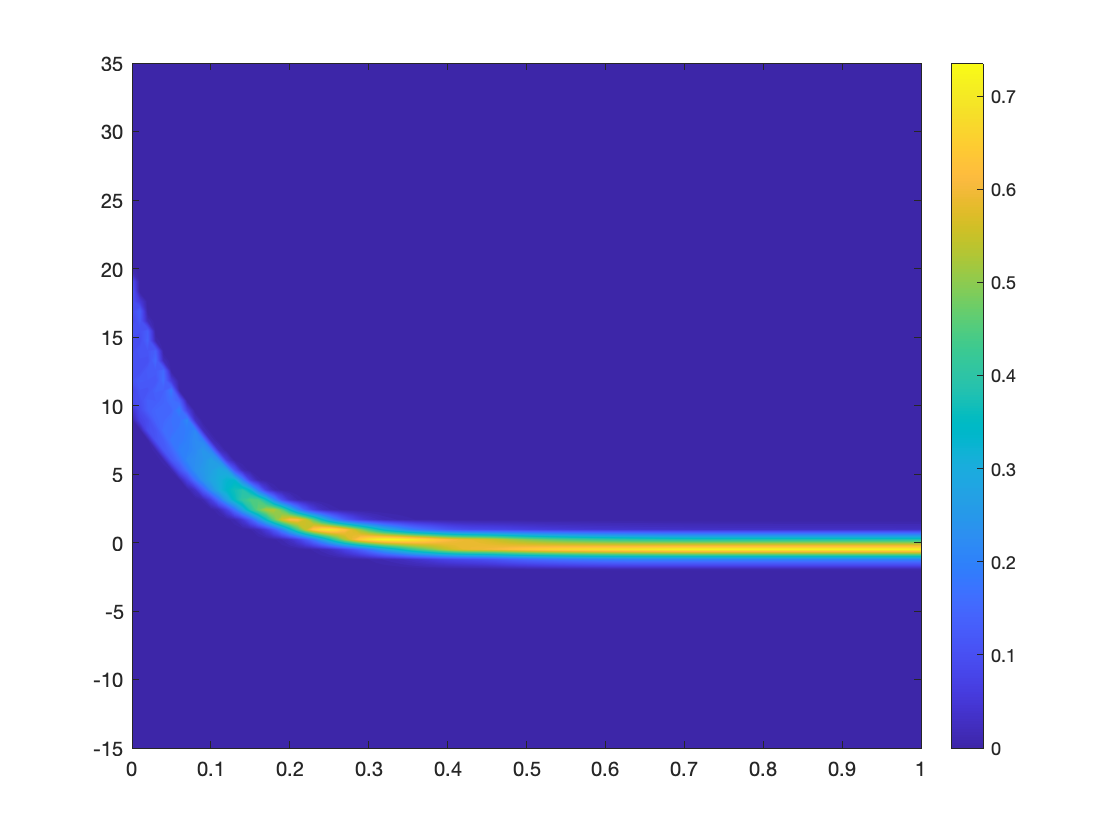}
				\caption{$	\mathbb{E}_\theta[f(t,v,\theta)]$ }
			\end{subfigure}
			\hfill
			\begin{subfigure}[b]{0.5\textwidth}
				\centering
				\includegraphics[width=\textwidth]{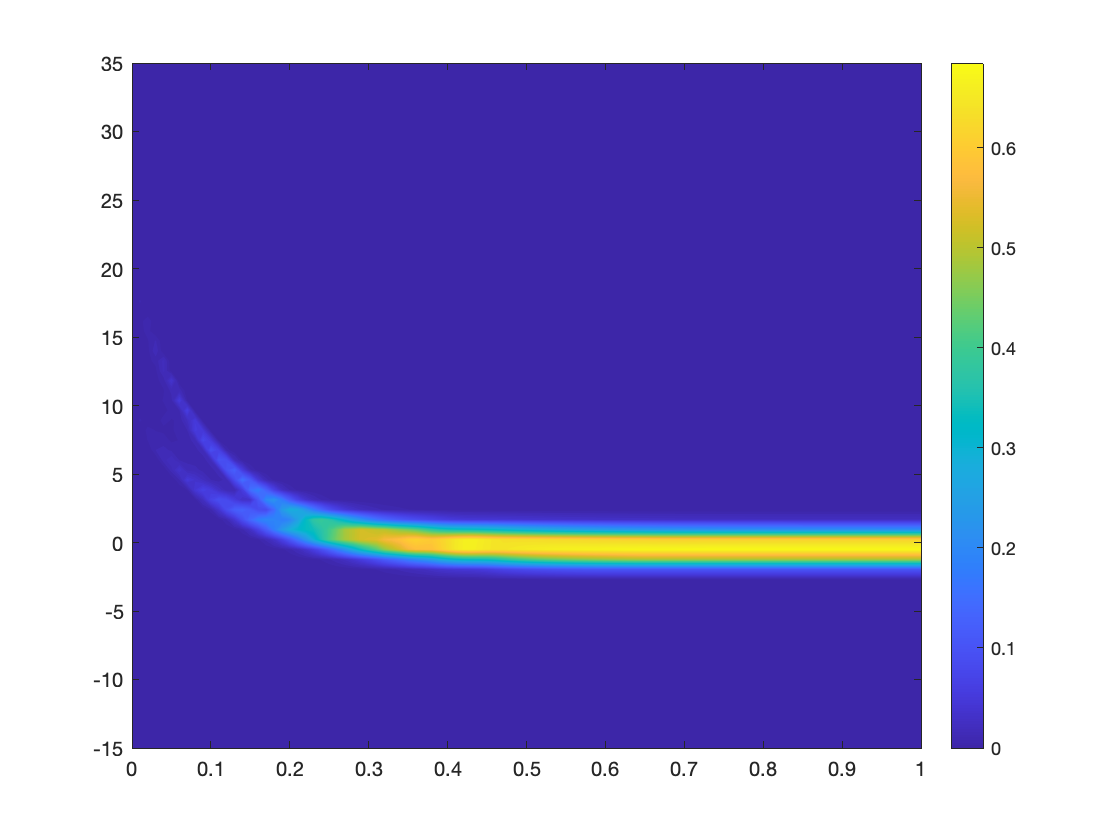}
				\caption{$	\sqrt{\mathbb{V}_\theta[f(t,v,\theta)]}$}
			\end{subfigure}
		\caption{{\em Test 3.}  Mean-field one-dimensional case. (a) Mean and  (b) standard deviation over time for the $\Hinf$ control in \eqref{control}, with parameters $\bar{p}= 1, \nu= 0.01, \rho_1 \sim \mathcal{N}(0,5), \rho_2\sim \mathcal{U}(-5,5)$.}\label{fig:mean1}
	\end{figure}
	
	\begin{figure}[H]
			\begin{subfigure}[b]{0.5\textwidth}
				\centering
				\includegraphics[width=\textwidth]{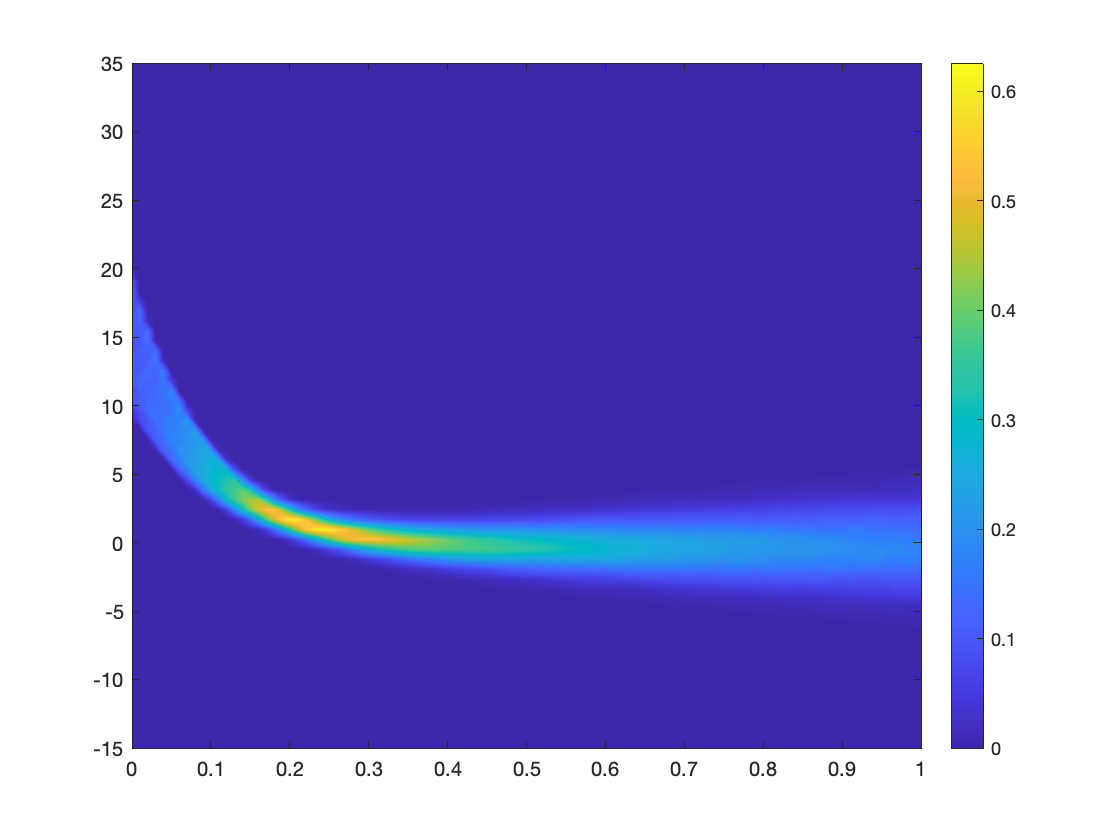}
				\caption{$	\mathbb{E}_\theta[f(t,v,\theta)]$ }
			\end{subfigure}
			\hfill
			\begin{subfigure}[b]{0.5\textwidth}
				\centering
				\includegraphics[width=\textwidth]{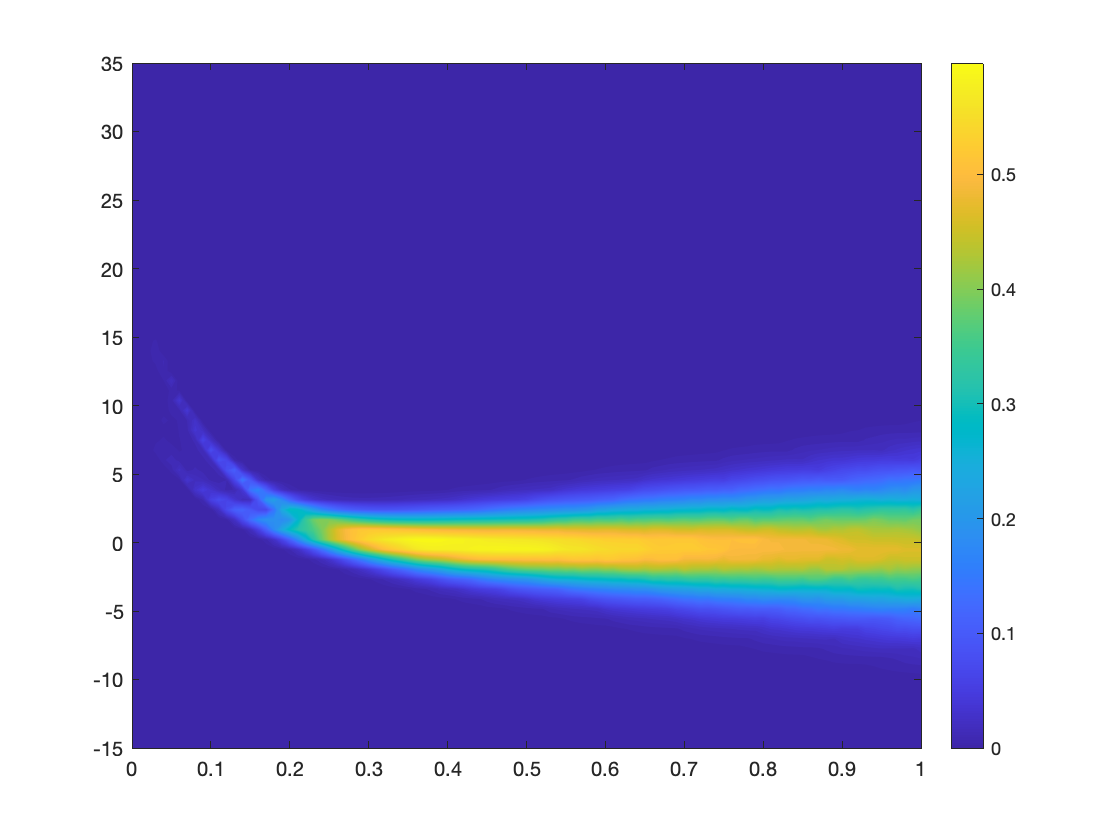}
				\caption{$		\sqrt{\mathbb{V}_\theta[f(t,v,\theta)]}$}
			\end{subfigure}
		\caption{{\em Test 3.}  Mean-field one-dimensional case. (a) mean and  (b) standard deviation over time for the averaged control in \eqref{eq:u_noiseless}, with parameters $\bar{p}= 1, \nu= 0.01, \rho_1 \sim \mathcal{N}(0,5), \rho_2\sim \mathcal{U}(-5,5)$.}\label{fig:mean2}
	\end{figure}
	We consider the same parameters as in Test 1 for the one dimensional microscopic case, where $\theta_1, \theta_2$ are uncertainties respectively with Gaussian distribution $\mathcal N(0,5)$ and uniform distribution $\mathcal U(-5,5)$. Hence we approximate the integrals in \eqref{eq:quad} using, respectively for $\theta_1$ and $\theta_2$, a Gauss-Hermite and Legendre-Gauss quadrature rules with $L = 40$ quadrature points.
	Figures \ref{fig:mean1} and \ref{fig:mean2} show a similar behavior with respect to the microscopic case in left plot of Figure \ref{fig:1D}, in particular we observe that less dispersion of the density for control of type \eqref{control}.

	\section{Conclusions}
	The introduction of uncertainties in multiagent systems is of paramount importance for the description of realistic phenomena.  Here we focused on the mathematical modelling and control of collective dynamics with random inputs and we investigated the robustness of controls proposing estimates based on $\Hinf$ theory in the linear setting.  Reformulating the control problem as a robust $\Hinf$  control problem, we derived sufficient conditions in terms of linear matrix inequalities (LMIs) to ensure the control performance, independently of the type of random inputs. Moreover, a robustness analysis is provided also in a mean-field framework, showing consistent results with the microscopic scale.
	Different numerical tests were proposed to compare the $\Hinf$ control with control synthesized minimizing the expectation of a function with respect to the random inputs. The numerical methods here developed make use of the stochastic Galerkin (SG) expansion for the microscopic dynamics while in the mean-field case we combine an SG expansion in the random space with a Monte Carlo method in the physical variables. The numerical experiments show that both controls are capable to drive the average particle trajectories towards a consensus state considering multiple sources of randomness and in different dimensions.
	We further observe that, in the $\Hinf$ setting, the variance is stabilized over time, this is not surprising since the $\Hinf$ control accounts for the random state in a feedback form, whereas in the noiseless control setting the uncertainty is averaged out. Nonetheless, these results confirm the goodness of the estimates for the control robustness for the uncertain dynamics. Further analysis is needed to extend these results in the $\Hinf$ setting to non-linear dynamics with uncertainities.  This can be studied for example introducing the so-called Hamilton-Jacobi-Isaacs equation, whose solution can be extremely challenging due to the high-dimensionality of multi-agent systems. 
	

	\appendix
	\numberwithin{equation}{section}
	\section{Averaged control}\label{app:noiseless}
	In this section we consider a control by  minimizing the expectation of the cost functional \eqref{eq:u_star} subject to the noisy model \eqref{noisyModel}. Hence, we consider the expected value of the quadratic cost
	\begin{equation*}
		\bar u^*(\cdot) = \arg\min_{u(\cdot)} \mathbb{E} \left[  \int_0^T  \frac{1}{2}( v^\top Q v + \nu u^\top R u)   \ dt \right] ,
	\end{equation*}
	where we introduce the matrices $Q=R= \frac{1}{N} \Id_N$. We claim that in this case an optimal feedback control is obtained as follows
	\begin{align}\label{eq:u_KS}
		\bar u(t) =  - \frac{N}{\nu} \Bigl( {K}  \ \mathbb{E}_{\theta}\left[  v \right] +  {S} \ \mathbb{E}_{\theta}\left[  \theta \right] \Bigr),
	\end{align}
	where ${K} \in \mathbb{R}^{N \times N}$ and  ${S} \in \mathbb{R}^{N \times Z}$ fulfill the Riccati matrix-equations
	\begin{equation}\label{eq:S-K}
		\begin{cases}
			-\dot{{K}} = {K}{A}+{A}^\top {K}-\frac{N}{\nu} K^2 + \frac{1}{N} I_N, \quad {K}(T) = 0_N,
			\\
			\
			\\
			-\dot{{S}} = K B+ A^\top S - \frac{N}{\nu} {K}{S} , \quad {S}(T) = 0_{N\times M}.
		\end{cases}
	\end{equation}

	\begin{thm}
		Assume matrices ${K}$ and $S$ have the following structures
		\begin{align*}
			({K})_{ij}=
			\begin{cases} 
				{k}_d,\qquad i=j,\\
				{{k}_o},\qquad i\neq j,
			\end{cases}
			\quad
			({S})_{ij}= s \cdot \mathds{1}_{N\times Z}.
		\end{align*}
		Matrices $K$ and $S$ are defined by $2$ and $1$ elements respectively.
		
		Then the $i-th$ component of the control $u$ is given by
		
		\begin{equation}\label{eq:u_KS_comp}
			u_i = -\frac{1}{\nu} \left( {k}_d \mathbb{E}\left[ v_i \right] + \frac{{k}_o}{N} \sum_{j \neq i}^N  \mathbb{E}\left[ v_j \right] + {s}\sum_{j = 1}^Z \mathbb{E}\left[ \theta_j \right] \right) ,
		\end{equation}
		where ${k}_d$ and ${k}_o$, after a scaling, turn out to be the same as in equations \eqref{kd_ko}, while ${s}$ satisfies
		
		\begin{equation} \label{eq:s}
			-\dot{{s}} = k_d + \alpha(N) k_o - \frac{s}{\nu } \left( k_d + \alpha(N) k_o \right), \qquad s(T)=0,
		\end{equation}
		with $\alpha(N) = \frac{N-1}{N}$.
	\end{thm}
	
	\begin{proof}
		From Proposition 2.1 of \cite{albi2021momentdriven}, we can prove that ${k}_d$, ${k}_o$ satisfy equations \eqref{kd_ko}.
		Given the structure of the matrices $S$, $K$ and $B$, and solving the second equation in \eqref{eq:S-K} componentwise leads to the following identities:
		\begin{align*}
			(KB)_{ij} &= k_d + (N-1)k_o,
			\\
			(A^\top S)_{ij} &= s\left( a_d + (N-1)a_o \right),
			\\
			(K S)_{ij} &= s\left( k_d + (N-1)k_o \right).
		\end{align*}
		We can further simplify the Riccati-matrix system \eqref{eq:S-K} using the dependency of coefficients $a_d,a_o$.
		And since we are interested in the dynamics of large number of agents, we introduce the following scalings
		\begin{equation}
			s \leftarrow N s, \quad {k}_d \leftarrow N k_d, \quad {k}_o \leftarrow N^2 k_o,\quad \alpha(N) = \frac{N-1}{N}.
		\end{equation}
		For the sake of simplicity, we keep the same notation also for the scaled variables $s, k_d, k_o$. Under this scaling, the second equation in \eqref{eq:S-K} reads 
		\[
		-\dot{s} = k_d + \alpha(N) k_o - \frac{s}{\nu}\left( k_d + \alpha(N) k_o \right),\qquad s(T)=0\,, 
		\]
		and the Riccati feedback law \eqref{eq:u_KS} is given by Eq. \eqref{eq:u_KS_comp}
		The coefficients $s(t), k_d(t), k_o(t)$, have to be determined integrating backwards in time.
	\end{proof}
	
	\begin{remark}
		In the  infinite horizon  case and without discount factor, \eqref{eq:s} reduces to
		\begin{align*}
			0 = k_d + \alpha(N) k_o - \frac{s}{\nu } \left( k_d + \alpha(N) k_o \right),
		\end{align*}
		hence s = $\nu$.
	\end{remark}

	\section{$\Hinf$ control setting}\label{app:Hinf}
	Define a state-space system $G: L^2 \rightarrow L^2$ by $y = G\theta$, with $\theta$ a random input as defined in \eqref{eq:theta_noise}, if
	\begin{align*}
		\ddt{v}(t) &= \widetilde{A} v(t) + \widetilde{B} \theta, \\
		y(t) &= C v(t) + D \theta,
	\end{align*}
	where $v(t)$ is the system state, and $y(t)$ is the observed output.
	It is proved (see e.g. \cite{dullerud2013course,boyd1994linear}) that for any stable state-space system, $G$, there exists a frequency transfer function $\hat{G} \in R\Hinf$ such that
	\begin{equation}\label{Ghat}
		\hat{G}(s) = \left[\begin{array}{@{}c|c@{}}
			\widetilde{A} & \widetilde{B} \\
			\hline
			C & D
		\end{array}\right] = D + C(s I_N-\widetilde{A})^{-1} \widetilde{B}.
	\end{equation}
	where $s$ is a complex number and
	$R\Hinf$ is the set of proper rational functions with no poles in the closed right half-plane, in particular $R\Hinf = R \cap \Hinf$, where $R$ is the space of rational functions and $\Hinf$ is a signal space of ``transfer functions'' for linear time-invariant systems, we refer to \cite{boyd1994linear,duan2013lmis, franklin2002feedback} for further theoretical details.
	
	State space $\widetilde{A},\widetilde{B},C,D$ or the transfer function is a representation of a system and these formats uses matrices or complex-valued functions (a signal) to parameterize the representation. 
	The signal norm $\Vert \cdot \Vert_{\Hinf}$ measures the size of the transfer function in a certain sense and the $\Hinf$-optimal control problem consists of finding a stabilizing controller $u = K y$ which minimizes the cost function 
	
	\begin{equation*}
		\Vert \hat{G} \Vert_{\Hinf} = \text{ess}\sup_{\omega \in \mathbb{R}} \bar{\sigma} (\hat{G}(i \omega)).
	\end{equation*}
	
	The direct minimization of the cost $\Vert \hat{G} \Vert_{\Hinf}$ turns out to be a very hard problem, and
	it is therefore not feasible to tackle it directly. Instead, it is much easier to construct
	conditions that state whether there exists a stabilizing controller which achieves the norm bound 
	
	\begin{equation*}
		\Vert \hat{G} \Vert_{\Hinf} \leq \gamma,
	\end{equation*}
	for a given $\gamma > 0$.
	
	
	The history of LMIs in the analysis of dynamical systems begins in about 1890, when Lyapunov published his seminal work introducing what we now call Lyapunov theory.
	One of the major next major contributions that we use in this work came in the early 1960’s, when Yakubovich, Popov, Kalman, and other researchers succeeded in reducing the solution of the LMIs to what we now call the positive-real (PR) lemma, that shows how LMIs can be used to constrain the eigenvalues of a system \cite{yakubovich1964solution,yakubovich1967method}.
	\begin{lem}\label{lem:LMI}
		Given the frequency transfer function $\hat{G}$, the following are equivalent:
		\begin{itemize}
			\item $\Vert \hat{G} \Vert_{\Hinf} \leq \gamma$.
			\item $\exists$ a positive definite square matrix of order $N$, $X> 0$ s.t. \eqref{lmi} holds.
			\begin{equation}\label{lmi}
				\begin{bmatrix}
					\widetilde{A}^\top X + X\widetilde{A} & X\widetilde{B}        \\
					\widetilde{B}^\top X      & -\gamma I_Z
				\end{bmatrix}
				+ \frac{1}{\gamma}
				\begin{bmatrix}
					C^\top \\
					D^\top
				\end{bmatrix}
				\begin{bmatrix}
					C & D
				\end{bmatrix}
				<
				0.
			\end{equation}
		\end{itemize}
	\end{lem}
	
	In \cite{willems1971least} the following Lemma with equivalent characterization through a Riccati equation as been established:
	
	\begin{lem}\label{ARElemma}
		The following are equivalent:
		\begin{itemize}
			\item $\exists X> 0$ s.t. \eqref{lmi} holds.
			\item $\exists X> 0$ s.t. \eqref{are} holds.
			\begin{equation}\label{are}
				\widetilde{A}^\top X + X\widetilde{A} - (X\widetilde{B} +  C^\top D) (-\gamma I_Z + \frac{1}{\gamma} D^\top D)^{-1} (\widetilde{B}^\top X + D^\top C) + \frac{1}{\gamma} C^\top C = 0.
			\end{equation}
		\end{itemize}
	\end{lem}
	\begin{proof}
		The structure of of Eq. \eqref{lmi} is
		\begin{equation*}
			\begin{bmatrix}
				\hat{A} & \hat{B}
				\\
				\hat{B}^\top & \hat{D}
			\end{bmatrix}
			\begin{bmatrix}
				x_1
				\\
				x_2
			\end{bmatrix}
			=
			\begin{bmatrix}
				0
				\\
				0
			\end{bmatrix},
		\end{equation*}
		this system can be solved using the Schur-complement theory.
		Provided that $\widehat{D}^{-1}$ exists, we have
		\begin{equation*}
			\begin{cases}
				\hat{A} x_1 + \hat{B}x_2 = 0 , \\
				\hat{B}^\top x_1 + \hat{D}x_2 = 0 ,
			\end{cases}
		\end{equation*}
		
		\[ \rightarrow \quad
		x_2 = - \hat{D}^{-1} \hat{B}^\top x_1
		\quad \rightarrow \quad
		\left( \hat{A} - \hat{B} \hat{D}^{-1} \hat{B}^\top \right) x_1 = 0.
		\]
		
		Hence, provided that exists $X$ such that \eqref{are} has a solution, then for all $ \xi$
		
		\[
		\xi^\top \left( \widetilde{A}^\top X + X\widetilde{A} - (X\widetilde{B} +  C^\top D) (-\gamma I_Z + \frac{1}{\gamma} D^\top D)^{-1} (\widetilde{B}^\top X + D^\top C) + \frac{1}{\gamma} C^\top C \right) \xi = 0.
		\]
		
		Further, \eqref{are} is the Schur-complement of
		
		\[
		M_{sc} :=
		\begin{bmatrix}
			\widetilde{A}^\top X + X\widetilde{A} +\frac{1}{\gamma}C^\top C
			&
			\widetilde{B}^\top X + D^\top C
			\\
			X\widetilde{B} + C^\top D
			&
			\frac{1}{\gamma}D^\top D - \gamma I_Z
		\end{bmatrix}.
		\]
		
		Hence for $\eta = -\hat{D}\hat{B}^\top \xi$ and $\forall \xi$, we have that 
		
		\[
		\begin{bmatrix}
			\xi & \eta
		\end{bmatrix}
		M_{sc}
		\begin{bmatrix}
			\xi \\ \eta
		\end{bmatrix} = 0.
		\]
		
		For a general vector $\left[ \xi \quad \varrho \right]^{\top}$, we compute
		
		\[
		\begin{bmatrix}
			\xi \\ \varrho
		\end{bmatrix}
		=
		\begin{bmatrix}
			\xi \\ \eta
		\end{bmatrix}
		+
		\begin{bmatrix}
			0 \\ \varrho - \eta
		\end{bmatrix},
		\]
		
		and then
		
		\[
		\begin{bmatrix}
			\xi & \varrho
		\end{bmatrix}
		M_{sc}
		\begin{bmatrix}
			\xi \\ \varrho
		\end{bmatrix}
		=
		0
		+
		\begin{bmatrix}
			0 & \varrho - \eta
		\end{bmatrix}
		M_{sc}
		\begin{bmatrix}
			0 \\ \varrho - \eta
		\end{bmatrix}
		=
		\begin{bmatrix}
			\varrho - \eta
		\end{bmatrix}^\top
		\begin{bmatrix}
			\frac{1}{\gamma}D^\top D - \gamma I_Z
		\end{bmatrix}
		\begin{bmatrix}
			\varrho - \eta
		\end{bmatrix}
		<
		0
		\]
		for $\gamma$ sufficiently large.
	\end{proof}
	
	
	\bibliographystyle{abbrv}
	\bibliography{biblioCS2}

\begin{thebibliography}{10}

\bibitem{bellomo20review}
G.~Albi, N.~Bellomo, L.~Fermo, S.-Y. Ha, J.~Kim, L.~Pareschi, D.~Poyato, and
  J.~Soler.
\newblock Vehicular traffic, crowds, and swarms: from kinetic theory and
  multiscale methods to applications and research perspectives.
\newblock {\em Math. Models Methods Appl. Sci.}, 29(10):1901--2005, 2019.

\bibitem{albi2021gradient}
G.~Albi, S.~Bicego, and D.~Kalise.
\newblock Gradient-augmented supervised learning of optimal feedback laws using
  state-dependent {R}iccati equations.
\newblock {\em IEEE Control Systems Letters}, 6:836--841, 2021.

\bibitem{albi2021momentdriven}
G.~Albi, M.~Herty, D.~Kalise, and C.~Segala.
\newblock Moment-driven predictive control of mean-field collective dynamics.
\newblock {\em SIAM Journal on Control and Optimization}, 60(2):814--841, 2022.

\bibitem{article_Albi_2015}
G.~Albi, L.~Pareschi, and M.~Zanella.
\newblock Uncertainty quantification in control problems for flocking models.
\newblock {\em Mathematical Problems in Engineering}, 2015, 04 2015.

\bibitem{ballerini2008empirical}
M.~Ballerini, N.~Cabibbo, R.~Candelier, A.~Cavagna, E.~Cisbani, I.~Giardina,
  A.~Orlandi, G.~Parisi, A.~Procaccini, M.~Viale, et~al.
\newblock Empirical investigation of starling flocks: a benchmark study in
  collective animal behaviour.
\newblock {\em Animal behaviour}, 76(1):201--215, 2008.

\bibitem{bacsar2008h}
T.~Ba{\c{s}}ar and P.~Bernhard.
\newblock {\em H-infinity optimal control and related minimax design problems:
  a dynamic game approach}.
\newblock Springer Science \& Business Media, 2008.

\bibitem{MR2974186}
N.~Bellomo and J.~Soler.
\newblock On the mathematical theory of the dynamics of swarms viewed as
  complex systems.
\newblock {\em Math. Models Methods Appl. Sci.}, 22(suppl. 1):1140006, 29,
  2012.

\bibitem{bongini2017inferring}
M.~Bongini, M.~Fornasier, M.~Hansen, and M.~Maggioni.
\newblock Inferring interaction rules from observations of evolutive systems i:
  The variational approach.
\newblock {\em Mathematical Models and Methods in Applied Sciences},
  27(05):909--951, 2017.

\bibitem{boyd1994linear}
S.~Boyd, L.~El~Ghaoui, E.~Feron, and V.~Balakrishnan.
\newblock {\em Linear matrix inequalities in system and control theory}.
\newblock SIAM, 1994.

\bibitem{carrillo2014derivation}
J.~A. Carrillo, Y.-P. Choi, and M.~Hauray.
\newblock The derivation of swarming models: mean-field limit and {W}asserstein
  distances.
\newblock In {\em Collective dynamics from bacteria to crowds}, pages 1--46.
  Springer, 2014.

\bibitem{carrillo2009double}
J.~A. Carrillo, M.~R. D'Orsogna, and V.~Panferov.
\newblock Double milling in self-propelled swarms from kinetic theory.
\newblock {\em Kinet. Relat. Models}, 2(2):363--378, 2009.

\bibitem{carrillo2010particle}
J.~A. Carrillo, M.~Fornasier, G.~Toscani, and F.~Vecil.
\newblock Particle, kinetic, and hydrodynamic models of swarming.
\newblock In {\em Mathematical modeling of collective behavior in
  socio-economic and life sciences}, pages 297--336. Springer, 2010.

\bibitem{CiCP-25-508}
J.~A. Carrillo, L.~Pareschi, and M.~Zanella.
\newblock Particle based g{PC} methods for mean-field models of swarming with
  uncertainty.
\newblock {\em Communications in Computational Physics}, 25(2):508--531, 2018.

\bibitem{CKPP19}
Y.-P. Choi, D.~Kalise, J.~Peszek, and A.~A. Peters.
\newblock A collisionless singular {C}ucker-{S}male model with decentralized
  formation control.
\newblock {\em SIAM J. Appl. Dyn. Syst.}, 18(4):1954--1981, 2019.

\bibitem{MR2165531}
S.~Cordier, L.~Pareschi, and G.~Toscani.
\newblock On a kinetic model for a simple market economy.
\newblock {\em J. Stat. Phys.}, 120(1-2):253--277, 2005.

\bibitem{MR3308728}
E.~Cristiani, B.~Piccoli, and A.~Tosin.
\newblock {\em Multiscale modeling of pedestrian dynamics}, volume~12 of {\em
  MS\&A. Model. Simul. Appl.}
\newblock Springer, Cham, 2014.

\bibitem{cucker2007emergent}
F.~Cucker and S.~Smale.
\newblock Emergent behavior in flocks.
\newblock {\em IEEE Trans. Automat. Control}, 52(5):852--862, 2007.

\bibitem{degond2007network}
P.~Degond, S.~G{\"o}ttlich, M.~Herty, and A.~Klar.
\newblock A network model for supply chains with multiple policies.
\newblock {\em Multiscale Model. Simul.}, 6(3):820--837, 2007.

\bibitem{dimarco2017uncertainty}
G.~Dimarco, L.~Pareschi, and M.~Zanella.
\newblock Uncertainty quantification for kinetic models in socio--economic and
  life sciences.
\newblock In {\em Uncertainty quantification for hyperbolic and kinetic
  equations}, pages 151--191. Springer, 2017.

\bibitem{duan2013lmis}
G.-R. Duan and H.-H. Yu.
\newblock {\em LMIs in control systems: analysis, design and applications}.
\newblock CRC press, 2013.

\bibitem{dullerud2013course}
G.~E. Dullerud and F.~Paganini.
\newblock {\em A course in robust control theory: a convex approach},
  volume~36.
\newblock Springer Science \& Business Media, 2013.

\bibitem{dyer2009leadership}
J.~R. Dyer, A.~Johansson, D.~Helbing, I.~D. Couzin, and J.~Krause.
\newblock Leadership, consensus decision making and collective behaviour in
  humans.
\newblock {\em Philos. Trans. Roy. Soc. B}, 364(1518):781--789, 2009.

\bibitem{d2006self}
M.~R. D’Orsogna, Y.-L. Chuang, A.~L. Bertozzi, and L.~S. Chayes.
\newblock Self-propelled particles with soft-core interactions: patterns,
  stability, and collapse.
\newblock {\em Phys. Rev. Lett.}, 96(10):104302, 2006.

\bibitem{Giselle}
G.~Estrada-Rodriguez and H.~Gimperlein.
\newblock Interacting particles with {L}\'{e}vy strategies: limits of transport
  equations for swarm robotic systems.
\newblock {\em SIAM J. Appl. Math.}, 80(1):476--498, 2020.

\bibitem{franklin2002feedback}
G.~F. Franklin, J.~D. Powell, A.~Emami-Naeini, and J.~D. Powell.
\newblock {\em Feedback control of dynamic systems}, volume~4.
\newblock Prentice hall Upper Saddle River, 2002.

\bibitem{Meurer}
G.~Freudenthaler and T.~Meurer.
\newblock P{DE}-based multi-agent formation control using flatness and
  backstepping: analysis, design and robot experiments.
\newblock {\em Automatica}, 115:108897, 13, 2020.

\bibitem{Garnier}
J.~Garnier, G.~Papanicolaou, and T.-W. Yang.
\newblock Consensus convergence with stochastic effects.
\newblock {\em Vietnam J. Math.}, 45(1-2):51--75, 2017.

\bibitem{goddard2022noisy}
B.~D. Goddard, B.~Gooding, H.~Short, and G.~Pavliotis.
\newblock Noisy bounded confidence models for opinion dynamics: the effect of
  boundary conditions on phase transitions.
\newblock {\em IMA Journal of Applied Mathematics}, 87(1):80--110, 2022.

\bibitem{MR2887663}
J.~G\'{o}mez-Serrano, C.~Graham, and J.-Y. Le~Boudec.
\newblock The bounded confidence model of opinion dynamics.
\newblock {\em Math. Models Methods Appl. Sci.}, 22(2):1150007, 46, 2012.

\bibitem{MR2425606}
S.-Y. Ha and E.~Tadmor.
\newblock From particle to kinetic and hydrodynamic descriptions of flocking.
\newblock {\em Kinet. Relat. Models}, 1(3):415--435, 2008.

\bibitem{han2017resolving}
Y.~Han, A.~Hegyi, Y.~Yuan, S.~Hoogendoorn, M.~Papageorgiou, and C.~Roncoli.
\newblock Resolving freeway jam waves by discrete first-order model-based
  predictive control of variable speed limits.
\newblock {\em Transportation Research Part C: Emerging Technologies},
  77:405--420, 2017.

\bibitem{MR2580958}
M.~Herty and L.~Pareschi.
\newblock Fokker-{P}lanck asymptotics for traffic flow models.
\newblock {\em Kinet. Relat. Models}, 3(1):165--179, 2010.

\bibitem{HePaSt15}
M.~Herty, L.~Pareschi, and S.~Steffensen.
\newblock Mean-field control and {R}iccati equations.
\newblock {\em Netw. Heterog. Media}, 10(3):699--715, 2015.

\bibitem{MR2844776}
M.~Herty and C.~Ringhofer.
\newblock Feedback controls for continuous priority models in supply chain
  management.
\newblock {\em Comput. Methods Appl. Math.}, 11(2):206--213, 2011.

\bibitem{hu2017uncertainty}
J.~Hu and S.~Jin.
\newblock Uncertainty quantification for kinetic equations.
\newblock In {\em Uncertainty quantification for hyperbolic and kinetic
  equations}, pages 193--229. Springer, 2017.

\bibitem{hu2015stochastic}
J.~Hu, S.~Jin, and D.~Xiu.
\newblock A stochastic galerkin method for {H}amilton--{J}acobi equations with
  uncertainty.
\newblock {\em SIAM Journal on Scientific Computing}, 37(5):A2246--A2269, 2015.

\bibitem{katz2011inferring}
Y.~Katz, K.~Tunstr{\o}m, C.~C. Ioannou, C.~Huepe, and I.~D. Couzin.
\newblock Inferring the structure and dynamics of interactions in schooling
  fish.
\newblock {\em Proceedings of the National Academy of Sciences},
  108(46):18720--18725, 2011.

\bibitem{khalil1996robust}
I.~Khalil, J.~Doyle, and K.~Glover.
\newblock {\em Robust and optimal control}.
\newblock prentice hall, new jersey, 1996.

\bibitem{le2010spectral}
O.~Le~Ma{\^\i}tre and O.~M. Knio.
\newblock {\em Spectral methods for uncertainty quantification: with
  applications to computational fluid dynamics}.
\newblock Springer Science \& Business Media, 2010.

\bibitem{lin2010robust}
P.~Lin and Y.~Jia.
\newblock Robust {H}-infinity consensus analysis of a class of second-order
  multi-agent systems with uncertainty.
\newblock {\em IET control theory \& applications}, 4(3):487--498, 2010.

\bibitem{liu2019robust}
J.~Liu, Y.~Zhang, H.~Liu, Y.~Yu, and C.~Sun.
\newblock Robust event-triggered control of second-order disturbed
  leader-follower mass: A nonsingular finite-time consensus approach.
\newblock {\em International Journal of Robust and Nonlinear Control},
  29(13):4298--4314, 2019.

\bibitem{liujia2011robust}
Y.~Liu and Y.~Jia.
\newblock Robust {H}-infinity consensus control of uncertain multi-agent
  systems with time delays.
\newblock {\em International Journal of Control, Automation and Systems}, 9, 12
  2011.

\bibitem{luo2021event}
Y.~Luo and W.~Zhu.
\newblock Event-triggered h-infinity finite-time consensus control for
  nonlinear second-order multi-agent systems with disturbances.
\newblock {\em Advances in Difference Equations}, 2021(1):1--19, 2021.

\bibitem{mo2013finite}
L.~P. Mo, H.~Y. Zhang, and H.~Y. Hu.
\newblock Finite-time {H}-infinity consensus of multi-agent systems with a
  leader.
\newblock In {\em Applied Mechanics and Materials}, volume 241, pages
  1608--1613. Trans Tech Publ, 2013.

\bibitem{motsch2014heterophilious}
S.~Motsch and E.~Tadmor.
\newblock Heterophilious dynamics enhances consensus.
\newblock {\em SIAM review}, 56(4):577--621, 2014.

\bibitem{KPAsurvey15}
K.-K. Oh, M.-C. Park, and H.-S. Ahn.
\newblock A survey of multi-agent formation control.
\newblock {\em Automatica}, 53:424--440, 2015.

\bibitem{mmpeet2020}
M.~M. Peet.
\newblock Lecture notes in {LMI} methods in optimal and robust control, 2020.

\bibitem{MR3157726}
A.~A. Peters, R.~H. Middleton, and O.~Mason.
\newblock Leader tracking in homogeneous vehicle platoons with broadcast
  delays.
\newblock {\em Automatica}, 50(1):64--74, 2014.

\bibitem{stern2018dissipation}
R.~E. Stern, S.~Cui, M.~L. Delle~Monache, R.~Bhadani, M.~Bunting, M.~Churchill,
  N.~Hamilton, H.~Pohlmann, F.~Wu, B.~Piccoli, et~al.
\newblock Dissipation of stop-and-go waves via control of autonomous vehicles:
  Field experiments.
\newblock {\em Transp. Research Part C: Emerging Techn.}, 89:205--221, 2018.

\bibitem{MR2247927}
G.~Toscani.
\newblock Kinetic models of opinion formation.
\newblock {\em Commun. Math. Sci.}, 4(3):481--496, 2006.

\bibitem{MR3948232}
A.~Tosin and M.~Zanella.
\newblock Kinetic-controlled hydrodynamics for traffic models with
  driver-assist vehicles.
\newblock {\em Multiscale Model. Simul.}, 17(2):716--749, 2019.

\bibitem{willems1971least}
J.~Willems.
\newblock Least squares stationary optimal control and the algebraic {R}iccati
  equation.
\newblock {\em IEEE Transactions on automatic control}, 16(6):621--634, 1971.

\bibitem{xiu2010numerical}
D.~Xiu.
\newblock {\em Numerical methods for stochastic computations}.
\newblock Princeton university press, 2010.

\bibitem{yakubovich1964solution}
V.~A. Yakubovich.
\newblock Solution of certain matrix inequalities encountered in non-linear
  control theory.
\newblock In {\em Doklady Akademii Nauk}, volume 156, pages 278--281. Russian
  Academy of Sciences, 1964.

\bibitem{yakubovich1967method}
V.~A. Yakubovich.
\newblock The method of matrix inequalities in the stability theory of
  nonlinear control systems, i, ii, iii.
\newblock {\em Automation and remote control}, 25(4):905--917, 1967.

\end{thebibliography}

	\vspace{1cm}
	\textit{E-mail address: }\href{mailto:giacomo.albi@univr.it}{giacomo.albi@univr.it} \\
	\textit{E-mail address: }\href{mailto:herty@igpm.rwth-aachen.de}{herty@igpm.rwth-aachen.de} \\
	\textit{E-mail address: }\href{mailto:segala@igpm.rwth-aachen.de}{segala@igpm.rwth-aachen.de}
	
\end{document}